\documentclass[10pt]{article}
\usepackage{graphicx}

\usepackage{blindtext}

\usepackage{graphicx}%
\usepackage{multirow}%
\usepackage{amsmath,amssymb,amsfonts}%
\usepackage{amsthm}%
\usepackage{mathrsfs}%
\usepackage{xcolor}%
\usepackage{textcomp}%
\usepackage{manyfoot}%
\usepackage{booktabs}%
\usepackage{algorithm}%
\usepackage{algorithmicx}%
\usepackage{algpseudocode}%
\usepackage{listings}%

\usepackage{hyperref}

\usepackage{bm}
\usepackage{enumerate}
\usepackage{enumitem}
\usepackage{bbm}
\usepackage{tikz}
\usepackage{tikz-cd}
\usepackage{pgf}
\usepackage{ifthen}

\usepackage[mode=buildnew]{standalone}

\usepackage{caption}
\usepackage{subcaption}

\usepackage{fullpage}

\DeclareMathSymbol{\Q}{\mathalpha}{AMSb}{"51}
\DeclareMathSymbol{\R}{\mathalpha}{AMSb}{"52}
\DeclareMathSymbol{\Z}{\mathalpha}{AMSb}{"5A}
\DeclareMathSymbol{\N}{\mathalpha}{AMSb}{"4E}
\DeclareMathSymbol{\C}{\mathalpha}{AMSb}{"43}

\newcommand{\G}{\ensuremath{{\mathbb{G}}}}

\newcommand{\supp}{\ensuremath{\text{\rm supp}}}

\newcommand{\X}{\ensuremath{{\mathcal{X}}}}

\newcommand{\A}{\ensuremath{\mathcal{A}}}

\newcommand{\Patt}[2]{\ensuremath{\mathcal{L}_{#2}(#1)}}
\newcommand{\Lang}[1]{\ensuremath{\mathcal{L}(#1)}}

\renewcommand{\vec}[1]{\mathbf{#1}}

\newcommand{\bu}{\mathbf{u}}
\newcommand{\bv}{\mathbf{v}}
\newcommand{\bt}{\mathbf{t}}

\newcommand{\bw}{\mathbf{w}}
\newcommand{\bd}{\mathbf{d}}
\newcommand{\be}{\mathbf{e}}

\newcommand{\bs}{\mathbf{s}}
\newcommand{\bz}{\mathbf{z}}

\newtheorem*{conjecture}{Conjecture}
\newtheorem*{theorem*}{Theorem}

\newtheorem{theorem}{Theorem}[section]
\newtheorem{lemma}[theorem]{Lemma}
\newtheorem{corollary}[theorem]{Corollary}

\newtheorem*{claim*}{Claim}

\theoremstyle{definition}
\newtheorem{definition}[theorem]{Definition}
\newtheorem{example}[theorem]{Example}
\newtheorem{remark}[theorem]{Remark}

\providecommand{\keywords}[1]
{
  \small	
  \textbf{\textit{Keywords---}} #1
}

\title{
An Algebraic Approach to Periodicity and Complexity of Delone Sets
}
\author{Pyry Herva (pyry.herva@gmail.com) and Jarkko Kari (jkari@utu.fi)}
\date{Department of Mathematics and Statistics, University of Turku, Finland 
\\
}

\begin{document}

\maketitle

\abstract{
\noindent
We study the connection between local complexity and periodicity of Delone sets using an algebraic approach based on annihilators, that is, finite linear relations between translations of a configuration. We model colored Delone sets (Delone multisets) as $\mathbb{R}^d$-configurations, that is, functions on $\mathbb{R}^d$ with finite range, and extend the notions of annihilators and periodizers from $\mathbb{Z}^d$-configurations to the geometric setting of Delone sets.
Our main results describe the structure imposed by the existence of non-trivial annihilators. We show that any $\mathbb{R}^d$-configuration admitting non-trivial annihilators has an annihilator of a particularly simple form, namely a product of finitely many difference operators corresponding to translations in pairwise rationally independent directions.
We further improve this result and prove that for configurations supported on Delone sets of finite local complexity (FLC), the corresponding directions can be chosen pairwise linearly independent over $\mathbb{R}$. As a consequence of the existence of the special annihilator, we show that $\mathbb{R}^d$-configurations with annihilators admit a decomposition into finitely many components that are periodic in pairwise rationally independent directions.
Moreover, we connect local complexity with algebraic structure by proving that Meyer sets with sufficiently slow patch complexity growth have low pattern complexity and hence admit non-trivial annihilators. Finally, we establish a sufficient and necessary condition under which the existence of suitable annihilators forces periodicity of $\R^d$-configurations supported on FLC Delone sets.
}

\noindent

\newpage
\tableofcontents

\newpage

\section{Introduction}

\subsection{Geometric background and motivation}

In crystalline materials, particles attach to each other to form ordered structures. Crystallography studies these materials from various perspectives \cite{crystal-structure-analysis, Senechal}. Traditional crystals exhibit periodic order, while quasicrystals are non-periodic. For a long time, it was believed that all ordered materials are necessarily periodic, until the discovery of quasicrystals by Shechtman in 1982 \cite{shechtman}. This discovery led to the development of the mathematical theory of aperiodic order, which studies
the mathematical foundation of quasicrystals, such as Delone sets and aperiodic tilings. 
The methods used in the area include tools from geometry, combinatorics, and dynamical systems
\cite{aperiodic1, aperiodic2}.

A central mathematical model for crystals and quasicrystals is given by Delone sets, which are uniformly discrete and relatively dense subsets of the Euclidean space.
They serve as a mathematical abstraction of the set of atomic positions of solid materials.
Many important examples, such as Penrose tilings \cite{penrose, Guyot1985AQS}, can be described in this framework.
Also, many of
the interesting spectral properties of quasicrystals can be formulated for Delone sets \cite{baake-lenz2, baake-lenz3}. 
Delone sets as an object of study in crystallography were introduced in the late 1930s and named after Russian mathematician Boris Delone \cite{borisdelone}.
In many applications, one considers colored Delone sets, where each point carries a type or label. These are commonly referred to as Delone multisets and can be naturally identified with functions on $\R^d$ with finite range whose supports are Delone sets. This allows us to combine geometric and algebraic viewpoints.

Delone sets can also be studied from a dynamical perspective, where one considers the action of translations on the space of all Delone sets. This leads to the notion of Delone dynamical systems, which have been extensively studied in the literature
\cite{baake-lenz2,baake-lenz3, lenz_stollmann, solomyak1998spectrum}.
A fruitful approach to study objects in aperiodic order and in particular Delone sets has turned out to be in terms of diffraction theory \cite{spectral_notions, baake2002weighted, Gour2005, hof, lee2002pure, meyer2016measures, Strungaru_2014}.
A typical mathematical tool in diffraction theory is Fourier analysis \cite{rudin2017fourier}.


One of the main ways to quantify local structure in Delone sets is through patch complexity, which counts the number of distinct patterns of points in bounded regions up to translation. 
More precisely, given a radius $T>0$ and a Delone set $S$, 
a $T$-patch of $S$ is the set of all points of $S$ within distance $T$ from some fixed point of $S$.
The patch complexity $N_S(T)$ gives the number of distinct $T$-patches of $S$ up to translation.
The patch complexity function $N_S(T)$ measures how rich the local structure of a Delone set $S$ is.
A Delone set has finite local complexity (FLC) if for every radius $T>0$, there are only finitely many distinct $T$-patches up to translation.
A fundamental problem is to understand when low local complexity enforces global regularity, in particular periodicity. 


\subsection{Motivation from multidimensional symbolic dynamics}

Questions of forced periodicity have been extensively studied in multidimensional symbolic dynamics.
In these considerations, one studies the connection between sufficiently low pattern complexity and the periodicity of a $\Z^d$-configuration, that is, a function on $\Z^d$ with finite range. 
Much is known, but there are still many open questions in the field.
The nature and difficulty of the questions depends much on the dimension in consideration.
In dimension $d=1$, the classical Morse Hedlund theorem \cite{morse-hedlund} states that a configuration is periodic if and only if it has sufficiently low factor complexity.
The exact threshold for low complexity is that for some $n$ the number of factors of length $n$ is at most $n$.

In higher dimensions, analogous questions are significantly more difficult. A prominent example is the still open Nivat’s conjecture \cite{Nivat}, which predicts that low rectangular pattern complexity forces periodicity in two dimensions.
More precisely, it states that
if a 2-dimensional configuration $c \in \A^{\Z^2}$ has at most $mn$ distinct $m \times n$ rectangular patterns for some $m$ and $n$, then it is necessarily periodic.
Althoug Nivat's conjecture remains an open problem in its full generality, many partial results towards it have been proved \cite{cyrkra, Cyr2016ComplexityPeriodicity, Mignosi, karimoutot, jarkkoetienneconference, icalp, fullproofs, Zamboni, sander-tijdeman2,  phdSzabados_Michal18a}.
Note that there is a more general version of Nivat's conjecture for convex shapes rather than just rectangles \cite{Cassaigne, dcfs}.
Note also that in dimensions $d >2$ the analogy of Nivat's conjecture fails \cite{Nivat3d}.

Another related problem is the periodic tiling problem, which concerns so called translational tilings and their periodicity \cite{grunbaum-shepard,Lagarias-wang, stein}.
For a reference on some recent research concerning the periodic tiling problem, see, \emph{e.g.},
\cite{bhattacharya, grebik-greenfeld-rozhon-tao, greenfeld2023tilingspectralityaperiodicityconnected, Greenfeld2021structure, greenfeld2023undecidabilitytranslationalmonotilings, greenfeld-tao-undecidable,greenfeld-tao2022}.
Both the periodic tiling problem and Nivat's conjecture
highlight deep connections between combinatorics, geometry, and dynamics, and many questions remain open.

A fruitful approach to the above problems is algebraic in nature \cite{icalp}.
In this approach, we 
study linear relations between translates of a configuration $c \in \A^{\Z^d}$.
More precisely, we assume that
$\A \subseteq \Z$
and consider linear expressions of the form 
$$
\sum_{i=1}^m a_i c (\bu - \bv_i)
$$
for fixed vectors $\bv_1,\ldots,\bv_m \in \Z^d$ and non-zero $a_1,\ldots,a_m \in \C$.
If any such linear combination vanishes for all $\bu \in \Z^d$, we say that $c$ has a non-trivial annihilator.
Equivalently, we code annihilators as Laurent polynomials $f \in \C[x_1^{\pm 1}, \ldots, x_d^{\pm 1}]$ and linear expressions of this type as discrete convolutions $f * c$ where $f$ is considered as a finitely supported function on $\Z^d$.

This algebraic formulation provides a convenient framework for studying periodicity. In particular, periodicity with period vector $\bv \in \Z^d$ corresponds to annihilation by a difference polynomial $X^{\bv}-1$.
More generally, annihilators impose algebraic constraints on the configuration and reflect underlying regularity. 
A key observation is that low pattern complexity implies the existence of non-trivial annihilators. 
More precisely, 
if for some finite shape $D \subseteq \Z^d$ there exist at most $|D|$ distinct patterns of shape $D$ in a configuration $c$, then $c$ has a non-trivial annihilator.
Furthermore, the structure of annihilators is highly constrained: any configuration admitting a non-trivial annihilator also admits an annihilator which is a product of finitely many difference polynomials.
Consequently, it follows that any configuration with annihilators can be decomposed to finitely many periodic components \cite{fullproofs}.


The integer grid $\Z^d$ is itself a Delone set, and $\Z^d$-configurations can be viewed as special cases of colored Delone sets, or Delone multisets. This suggests that algebraic methods developed for symbolic dynamics may extend to more general Delone sets in $\R^d$. 
However, such an extension is not trivial. In the absence of an underlying lattice structure, additional geometric features come into play. 
In particular, finite local complexity (FLC)
plays a crucial role in determining whether algebraic methods remain applicable and what conclusions can be drawn.

\subsection{Our contributions}

In this paper, we develop an algebraic approach to Delone configurations (equivalently, Delone multisets) that connects local complexity, algebraic structure, and periodicity.
We expand the concept of annihilators to $\R^d$-configurations in a natural way and study different settings where annihilators arise. A dilation lemma (Lemma~\ref{lem:Rdilationlemma}) establishes the presence of long-range correlations in those Delone configurations that admit annihilators. While long-range order in Delone sets is traditionally identified with a pure point diffraction spectrum~\cite{aperiodic1, aperiodic2}, our work offers a complementary algebraic perspective on aperiodic order based on annihilation.

Our main results are related to the structure of annihilators of $\R^d$-configurations.
Building on the corresponding results of $\Z^d$-configurations discussed above, we show that if an $\R^d$-configuration admits a non-trivial annihilator, then it also admits an annihilator which is a product of finitely many difference operators in pairwise rationally independent directions (Theorem \ref{main result 1}).
However, this extension of the $\Z^d$-setting is far from trivial.
In the absence of an underlying lattice structure, the algebraic tools available in the $\Z^d$-setting do not apply directly.
Our proof uses the known $\Z^d$-result in a non-trivial way, combining it with geometric arguments to transfer structural information from the discrete setting of $\Z^d$-configurations to the continuous setting of $\R^d$-configurations.
For configurations supported on Delone sets of finite local complexity, we further refine this result by showing that the corresponding directions of the difference operators can be chosen linearly independent over $\R$, not just over $\Q$ (Theorem \ref{main result 2}).
Analogously as in the case of $\Z^d$-configuration, it follows from the existence of the special annihilator, that $\R^d$-configurations with annihilators can be decomposed into finitely many periodic components (Theorem \ref{thm:generalized decomposition theorem}).

We also investigate conditions under which annihilators arise in the Delone setting. 
In particular, we consider Delone sets with additional geometric structure, focusing on Meyer sets, that is, Delone sets $S$ whose difference set $S-S$ is also a Delone set. 
We show that for Meyer sets, sufficiently slow growth of patch complexity (namely if $N_S(T) = o(T^d)$) implies low pattern complexity (Theorem \ref{thm: low complexity Meyer}) and hence the existence of non-trivial annihilators (by a direct extension of the corresponding result for $\Z^d$-configurations that low complexity implies the existence of annihilators).
This result establishes a direct link between local geometric complexity and algebraic structure.
Our result provides a corrected formulation (Theorem \ref{thm: meyer growth bound}) of a conjecture by Lagarias and Pleasants \cite{lagarias_pleasants_ETDS} that was proven false by Cassaigne \cite{cassaigne-example}.

Finally, we establish a theorem of forced periodicity for Delone configurations of finite local complexity. 
We give a characterization for conditions on annihilators to imply global strong periodicity for $\R^d$-configurations supported on Delone sets of finite local complexity (Theorem \ref{thm: forced periodicity Delone FLC}). 
Combined with the previous results, this result shows that sufficiently strong constraints on local patterns give rise to annihilators, whose structure in turn forces global periodic order.

\subsection{The structure of the paper}

The paper is organized as follows. 
In Section \ref{sec: preliminaries} we introduce our basic notation and definitions concerning configurations and complexity. 
Section \ref{sec: symbolic dynamics} reviews relevant concepts from multidimensional symbolic dynamics and the algebraic approach based on annihilators. 
In Section \ref{sec: delone sets} we introduce Delone sets and extend the algebraic framework concerning $\Z^d$-configurations to $\R^d$-configurations. 
Section \ref{sec: delone configurations with annihilators} presents
our structural results on $\R^d$-configurations with annihilators, including a periodic decomposition.
In Section \ref{sec: meyer sets} we study Meyer sets with slow patch complexity growth and show that sufficiently slow growth implies low pattern complexity and hence the existence of non-trivial annihilators.
In Section \ref{sec: forced periodicity FLC} we prove a sufficient and necessary condition of forced periodicity for FLC Delone configurations.
We conclude with Section \ref{sec: conclusion} by recalling the main contributions of the paper and giving some further directions for future research around the topic.


\section{Preliminaries} \label{sec: preliminaries}

In this section, we introduce the general framework of configurations on $\Z^d$ and $\R^d$, including notions of periodicity, pattern complexity, and algebraic structure in terms of convolution, which will be central in our later analysis of annihilators.
We present the basic notation and definitions used throughout the paper.

\subsection{Notation}

As usual, we denote by $\Z$, $\Q$, $\R$, and $\C$ the sets of integers, rational numbers, real numbers, and complex numbers, respectively.
By $\Z_+$ and $\R_+$ we mean the sets of positive integers and positive reals, respectively.
Moreover, the set of natural numbers is the set $\N= \Z_+ \cup \{0\}$ of non-negative integers.
We consider the sets $\Z^d$ and $\R^d$ where $d$ is a positive integer --- the dimension.

For any two sets $A$ and $B$, we denote $A \Subset B$ if $A \subseteq B$ and $|A| < \infty$.
By $A^B$ we mean the set of all functions from $B$ to $A$.

For a vector $\bu \in \R^d$, we denote by $||\bv||$ its Euclidean norm.
Then $||\bu-\bv||$ is the Euclidean distance of $\bu,\bv\in\R^d$.
We denote by $B_T(\vec{u})$ and $B_T^{\circ}(\vec{u})$ the closed and open Euclidean balls, respectively, of radius $T$ centered at $\vec{u} \in \R^d$ for any real number $T \geq 0$.
In addition, we use the shorthand notations $B_T = B_T(\vec{0})$ and $B_T^{\circ} = B_T^{\circ}(\vec{0})$.

We denote by $\be_i \in \Z^d$ the vector whose $i$th coordinate is 1 and all the other coordinates are 0.
The scalar product of two vectors $\bu, \bv \in \R^d$ is denoted by $\bu \cdot \bv$.
For any vectors $\bv_1,\ldots,\bv_k \in \R^d$, we denote
by $L(\bv_1,\ldots,\bv_k) = \{ a_1\bv_1 + \ldots+a_k \bv_k \mid a_1,\ldots,a_k \in \R \}$
the linear subspace of $\R^d$ spanned by them.
Similarly for any (possibly infinite) set $S\subseteq \R^d$, we denote by $L(S)$ the subspace generated by $S$.
On the other hand, by $\langle \bv_1,\ldots,\bv_k \rangle$ we mean the subgroup of $\R^d$ generated by the vectors $\bv_1,\ldots,\bv_k$.
We also use the notation $\Z[S]$ to denote the abelian group generated by the set $S$.

\subsection{Configurations and periodicity}

We now define configurations and introduce different notions of periodicity.
These are concepts that will play a central role in our study.
In the following, $d$ denotes the dimension, and
by $G$ we mean either
$\Z^d$ or $\R^d$.

Let $\A \subseteq \C$ be any subset of complex numbers.
In this paper 
we study functions $c \colon G \rightarrow \A$, that is, elements of the set $\A^G$.

\begin{definition}
If $\A \Subset \C$, then $c \in \A^{G}$ is called a \emph{$G$-configuration}.
The set $\A^{G}$ of all $G$-configurations over $\A \Subset \C$ is called the \emph{$G$-configuration space} over $\A$.
Configurations may also be called \emph{finitary} to emphasize the fact that their range is finite.
\end{definition}

\noindent
If it is clear from the context what $G$ is, we may call any $G$-configuration simply a configuration.

\begin{remark}
Notice the distinction between the concepts of configurations and general functions of the
set $\C^{G}$. 
Indeed, the range of a general such function is allowed to be infinite while the range
of a configuration is always finite.
\end{remark}

\begin{definition}
We say that $c \in \A^G$ is \emph{integral} if $\A \subseteq \Z$.
\end{definition}


\noindent
We often use the notation $c_{\vec{u}} = c(\bu)$ for the value of $c \in \A^G$ at position $\bu$.

\begin{definition}
The \emph{support} of $c$ is the set
$$
\supp(c) = \{ \bu \in G \mid c_{\bu} \neq 0 \}.
$$
\end{definition}

\begin{definition}
The \emph{translation} $\tau^{\vec{t}}(c)$ of $c \in \A^G$ by $\vec{t} \in G$ is defined such that $\tau^{\vec{t}}(c)_{\bu} = c_{\vec{u} -\vec{t}}$ for all $\vec{u} \in G$.
We say that $c$ is \emph{$\vec{t}$-periodic} if $\tau^{\mathbf{t}}(c)=c$, and we say that it is \emph{periodic} if it is $\vec{t}$-periodic for some non-zero $\mathbf{t}\in G$.
If $c$ is $\bt$-periodic, we call $\bt$ a \emph{period} or a \emph{period vector} of $c$.  
Moreover, 
$c$
is \emph{strongly periodic} if it has $d$
linearly independent vectors of periodicity over $\R$. 
Finally, $c$ is \emph{periodic in direction} $\bv$ if it is $k \bv$-periodic for some non-zero $k \in \R$.
\end{definition}

The indicator function $\mathbbm{1}_S \colon \R^d \rightarrow \{0,1\}$ of a subset $S \subseteq \R^d$ is defined as usual:
$$
\mathbbm{1}_S(\vec{u}) =
\begin{cases}
    1 &\text{, if } \vec{u} \in S \\
    0 &\text{, if } \vec{u} \not \in S
\end{cases}.
$$
A set $S \subseteq \R^d$ is $\vec{t}$-periodic, periodic, 
strongly periodic or periodic in direction $\bv$ if its indicator function is $\vec{t}$-periodic, periodic, 
strongly periodic or periodic in direction $\bv$, respectively.

\subsection{Patterns and pattern complexity}

To study local structure, we consider patterns of configurations on finite subsets of $G$. The number of such patterns, measured by pattern complexity, will be our main tool in relating local constraints to global behavior.
Let us give precise definitions for these concepts.

\begin{definition}
We call any non-empty finite set $D\Subset G$ a \emph{$G$-shape}.
\end{definition}

\noindent
In particular, $\Z^d$-shapes are also $\R^d$-shapes.
If $G$ is known from the context, we may call $G$-shapes simply shapes.

\begin{definition}
For a shape $D \Subset G$, a function $p \colon D \rightarrow \A$ is called a \emph{$G$-pattern} of shape $D$ or just a pattern.
It may also be called a \emph{$D$-pattern} if it is relevant to emphasize the shape of the pattern.
\end{definition}

\noindent
For a fixed shape $D \Subset G$, the set of all $D$-patterns of a configuration $c \in \A^G$ is the set $\Patt{c}{D} = \{ \tau^{\vec{t}}(c)|_D \mid \vec{t} \in G \}$.
The set of all patterns of $c$
is denoted by $\Lang{c}$ which we call the \emph{language of $c$}. 
For a set $\mathcal{S} \subseteq \A^G$ of configurations,
we define $\Patt{\mathcal{S}}{D}$ and  $\Lang{\mathcal{S}}$ as the unions of $\Patt{c}{D}$ and $\Lang{c}$, respectively, over all
$c\in \mathcal{S}$.

We measure the diversity of local patterns using the following notion of pattern complexity.

\begin{definition}
The \emph{pattern complexity} $P_c(D)$ of a configuration $c \in \A^{G}$ with respect to a shape $D \Subset G$ is the number of distinct $D$-patterns that $c$ contains, that is, $P_c(D) = |\Patt{c}{D}|$.
\end{definition}

\noindent
To emphasize whether we are dealing with $\Z^d$-configurations or $\R^d$-configurations, we may use the terms \emph{$\Z^d$-pattern complexity} or \emph{$\R^d$-pattern complexity}, respectively.

\begin{definition}\label{def: low complexity}
A configuration $c \in \A^G$ has \emph{low complexity} with respect to a shape $D \Subset G$ if
$$
P_c(D) \leq |D|.
$$
\end{definition}







The pattern complexity of a set $S \subseteq \R^d$ with respect to a shape $D \Subset \R^d$ is $P_S(D) = P_{\mathbbm{1}_S}(D)$.
Note that we have
$$
P_S(D) = | \{ \left (S \cap (D + \bt) \right ) - \bt \mid \bt \in \R^d \}| = |\{(S- \bt) \cap D \mid \bt \in \R^d\}|.
$$
By a slight abuse of terminology we call any set $D' \Subset \R^d$ a \emph{$D$-pattern of the set $S$} if $D' = S\cap(D+\bt)$ for some $\bt$. 
Thus, $P_S(D)$ counts 
the number of $D$-patterns of $S$ up to translation.
We denote by
$\Patt{S}{D} = \{(S \cap (D+\bt))-\bt \mid \bt \in \R^d \}$ the set of all $D$-patterns of $S$ translated to origin.

\begin{remark}
Above we defined pattern complexity only for configurations.
For non-finitary functions $c \in \C^G$, the pattern complexity would be always infinite.
\end{remark}

\subsection{Polynomials and $\R^d$-polynomials}

To study linear relations between translations of configurations, we introduce an algebraic framework based on discrete convolution. This allows us to encode such relations in terms of finitely supported functions which can be interpreted as Laurent polynomials in the discrete setting.

\begin{definition}
A function $f \in \C^G$ is \emph{finitely supported} if its support is a finite set.
\end{definition}

\begin{definition}
The (discrete) \emph{convolution} $f * g$ of a finitely supported function $f \in \C^G$ and any function $g \in \C^G$ is defined such that
$$
(f * g)(\bu) = \sum_{\bv \in \supp(f)} f(\bv) g(\bu - \bv)
$$
for all $\bu \in G$.
\end{definition}

\noindent
Convolution may be interpreted as a finite linear combination of translates of $g$, where the coefficients are given by $f$.

In the case $G = \Z^d$,
finitely supported functions are identified with Laurent polynomials.
More precisely, any finitely supported function $f\in \C^{\Z^d}$
is written as 
$$
f = f(X) = \sum_{\bu = (u_1,\ldots,u_d) \in \supp(f)} f_{\bu} x_1^{u_1}\cdots x_d^{u_d} = \sum_{\bu \in \supp(f)} f_{\bu} X^{\bu}
$$
where $X = (x_1,\ldots,x_d)$ where we have used the standard notation $X^{\bu} = x_1^{u_1} \cdots x_d^{u_d}$ for $\bu = (u_1,\ldots,u_d)$.
Conversely, every Laurent polynomial corresponds to a finitely supported function in this way. Under this identification, convolution corresponds to multiplication of polynomials and Laurent series.
We usually drop the word ``Laurent'' and instead talk simply about polynomials even though we mean Laurent polynomials.
As usual, we denote by $\C[X^{\pm1}] = \C[x_1^{\pm1},\ldots,x_d^{\pm1}]$ the set of all complex polynomials in $d$ variables $X=(x_1,\ldots,x_d)$.
The set of all integral polynomials is denoted by $\Z[X^{\pm1}]$.

In the case $G=\R^d$, there is no direct polynomial ring structure analogous to the discrete case.
However, we will still consider finitely supported functions $f\in\C^{\R^d}$ and use the same discrete convolution operation.

\begin{definition}
An \emph{$\R^d$-polynomial} is any finitely supported function $f \in \C^{\R^d}$.
\end{definition}

\noindent
We use the convenient ``polynomial notation'' also for $\R^d$-polynomials, that is, we denote
$$
f=f(X) = \sum_{\bu \in \supp(f)} f_{\bu} X^{\bu},
$$
keeping in mind that this is purely formal. This notation allows us to treat convolution expressions in a unified way across both $\Z^d$ and $\R^d$.

\section{Symbolic dynamics} \label{sec: symbolic dynamics}

In this section, we review concepts from multidimensional symbolic dynamics that are relevant to our work, with particular emphasis on the algebraic approach based on annihilators. This framework provides the foundation for our results, allowing us to relate local complexity to algebraic structure and ultimately to periodicity.

\subsection{Basics}

In the following, we briefly recall the standard setting of multidimensional symbolic dynamics.
For a more thorough reference, see, {\it e.g.}, \cite{tullio, kurka,lindmarcus}.

The configuration space $\A^{\Z^d}$ is considered as a compact metric space as a product topology of the finite alphabet $\A$.
A subset $\mathcal{S} \subseteq \A^{\Z^d}$ of the configuration space is a \emph{subshift} if it is topologically closed and translation-invariant meaning that if $c \in \mathcal{S}$, then for all $\vec{t} \in \Z^d$ also $\tau^{\vec{t}}(c) \in \mathcal{S}$.
Equivalently, subshifts can be defined by using forbidden patterns.

The \emph{orbit} of a configuration $c$ is the set $\mathcal{O}(c) = \{ \tau^{\vec{t}}(c) \mid \vec{t} \in \Z^d \}$ of its every translate.
The
\emph{orbit closure} $\overline{\mathcal{O}(c)}$ is the topological closure of its orbit under the prodiscrete topology.
It is the smallest subshift that contains $c$.
It consists of all configurations $c'$ such that $\Lang{c'}\subseteq \Lang{c}$.

\subsection{An algebraic approach to multidimensional symbolic dynamics}

An important tool in our work is an algebraic approach to multidimensional configurations, developed by the second author and Szabados \cite{icalp}. 
In this setting, we encode configurations as formal power series and study linear relations between their translates considering convolution by polynomials.

In this approach, a function $c \in \C^{\Z^d}$ is represented as the formal power series
$$
c(X) 
= 
\sum_{\bu \in \Z^d} c_{\bu} X^{\bu}
$$
in $d$ variables $X=(x_1,\ldots,x_d)$.

The product $fc$ of a (Laurent) polynomial $f \in \C[X^{\pm1}]$ and a formal power series $c \in \C[[X^{\pm1}]]$ is defined naturally such that its coefficient at $\bu$ is
$$
(fc)_{\bu} = \sum_{\bv \in \supp(f)} f_{\bv} c_{\bu - \bv}.
$$
Thus, $fc$ corresponds to the convolution $f * c$ where $f$ and $c$ are regarded as functions in $\C^{\Z^d}$ with the additional assumption that $f$ is finitely supported.


Next, we represent our two key definitions.

\begin{definition}
A polynomial $f$ \emph{annihilates} (or is an annihilator of) a function $c \in \C^{\Z^d}$ if $fc=0$.
\end{definition}

\noindent
Annihilators encode linear relations between translates of a configuration, and provide a convenient algebraic framework for studying periodicity and structural properties.
Clearly, $c$ is $\bv$-periodic if and only if it 
is annihilated by the difference polynomial $X^{\bv}-1$.
More general annihilators may capture more subtle forms of regularity.

\begin{definition}
A polynomial $f$ \emph{periodizes} (or is a \emph{periodizer} of)  $c$ if $fc$ is strongly periodic. 
\end{definition}

\noindent
Clearly, $c$ has a non-trivial (= non-zero) annihilator if and only if it has a non-trivial periodizer.
Indeed, any annihilator is a periodizer and conversely if $c$ has a non-trivial periodizer $f$, then it has a non-trivial annihilator $(X^{\bv}-1)f$ for some $\bv \neq \vec{0}$.

\begin{remark}
If an integral configuration $c \in \Z^{\Z^d}$ has a non-trivial annihilator $f \in \C[X^{\pm1}]$, then it has an integral annihilator $f' \in \Z[X^{\pm1}]$ that has the same support as $f$ \cite{DLT_invited}.
\end{remark}

\begin{remark}
  In this paper the coefficients of configurations and hence the formal power series presenting them are complex numbers.
  However,
  this is not necessarily always the case.
  The algebraic approach can be applied to more general settings.
  The coefficients of the configurations may belong to any abelian group $M$.
  Then the coefficients of the polynomials belong to a ring $R$ such that $M$ is an $R$-module.
  Under these assumptions the convolution multiplication is defined.
  For example, in \cite{DLT_special_issue} we studied a setting with $M = \Z^n$ and $R = \Z^{n \times n}$.
\end{remark}

\subsubsection*{Line polynomials}

A special class of annihilators that plays an important role is given by line polynomials, which correspond to algebraic relations supported along a single direction.

\begin{definition}
A polynomial $f$ is a \emph{line polynomial} if its support contains at least two points, and it is contained in a line, that is, $\supp(f) \subseteq \bu + \Q \bv$ for some $\bu, \bv \in \Z^d$.
In this case we say that $f$ is a \emph{line polynomial in direction} $\bv$ and call $\bv$ a \emph{direction} of $f$.
\end{definition}
\noindent
Clearly, any line polynomial in direction $\bv$ is also a line polynomial in any parallel non-zero direction $\bv'$ over $\Q$.

Non-trivial difference polynomials are line polynomials. 
Hence, periodicity of a configuration implies annihilation by a line polynomial.
In fact, also the converse is quite easily seen to be true. 
Indeed, annihilation by a line polynomial of a configuration defines a recurrence relation on the configuration which implies periodicity due to the finiteness of $\A$.
Thus, a configuration is periodic if and only if it has a line polynomial annihilator.
In particular, any one-dimensional configuration with a non-trivial annihilator is periodic.
Analogous properties will fail in the $\mathbb{R}^d$-setting without additional assumptions which motivates our later refinements.

\subsection{Some known results}

Now, let us review known results from the algebraic approach to multidimensional symbolic dynamics that will be necessary to our work.
These results establish a connection between low pattern complexity, the existence of annihilators, and the structure of configurations.

\subsubsection*{Special annihilators}

First, the low complexity assumption $P_c(D) \leq |D|$ implies the existence of non-trivial periodizers and annihilators.

\begin{lemma}[\cite{icalp}]
  \label{th:low_complexity}
  Let $c \in \A^{\Z^d}$ be a configuration that has low complexity with respect to shape $D = \{\bd_1,\ldots,\bd_m\} \Subset \Z^d$, that is, $P_c(D) \leq |D|$.
  Then $c$ has a non-trivial annihilator. 
  More precisely, $c$ has a periodizer of the form 
  $$
  a_1 X^{-\bd_1} + \ldots + a_m X^{-\bd_m}
  $$
   for some non-zero $(a_1,\ldots,a_m) \in \C^m$.
\end{lemma}


If annihilators exist, then annihilators of very particular type exist. 
Indeed, annihilators can be chosen in a simple product form involving difference polynomials as the following theorem states.


\begin{theorem}[\cite{icalp}, \cite{fullproofs}] \label{special annihilator configurations}
  Let $c \in \A^{\Z^d}$ be an integral configuration with a non-trivial integral annihilator $f$.
  Then for all $\bu \in \supp(f)$ the configuration $c$ has an annihilator of the form
  $$
  (X^{\vec{v}_1} - 1) \cdots (X^{\vec{v}_m} - 1)
  $$
  where the vectors $\vec{v}_1, \ldots , \vec{v}_m$ are pairwise linearly independent and each $\bv_i$ is parallel to $\bu_i - \bu$ for some $\bu_i \in \supp(f) \setminus \{\bu\}$.
\end{theorem}

The proof of the above theorem is based on applying Hilbert's Nullstellensatz.
A \emph{dilation} of a polynomial $f(X)$ is a polynomial of the form $f(X^k)$ for some integer $k$.
The following lemma is crucial in the proof of Theorem \ref{special annihilator configurations} and in our forthcoming considerations.

\begin{lemma}[The dilation lemma \cite{fullproofs}] 
  Let $c \in \A^{\Z^d}$ be an integral configuration, and let $f$ be a non-trivial 
  integral annihilator of $c$. 
  There exists a positive integer $r$ such that for every positive integer $k$ with $\gcd(k,r)=1$ also $f(X^k)$ annihilates $c$.
\end{lemma}

\noindent
Let us call a number $r$ that has the property of the above lemma a \emph{dilation constant} of $c$ with respect to $f$.
The dilation lemma allows one to derive structured annihilators of a specific form.
More precisely, the following lemma is proved in the proof of Theorem \ref{special annihilator configurations}.

\begin{lemma}[\cite{fullproofs}]
\label{lemma1}
 Let $c \in \A^{\Z^d}$ be an integral configuration with a non-trivial integral annihilator $f$.
  Then for all $\bu \in \supp(f)$ the configuration $c$ is annihilated by the polynomial
$$
\prod_{\bv \in \supp(f) \setminus \{\bu\}} (X^{r(\bv-\bu)}-1)
$$ 
where $r$ is a dilation constant of $c$ with respect to $f$.
\end{lemma}

\noindent
From the proof of the dilation lemma we adapt the following result.
For completeness, we provide a short description of the proof.

\begin{lemma}[Adapted from \cite{fullproofs}] \label{lemma2}
  Let $\mathcal{I}$ be an arbitrary index set, and let $\left ( c^{(i)} \right ) _{i \in \mathcal{I}}$ be a collection of integral configurations over the same alphabet $\A \Subset \Z$.
  If $f$ is a non-trivial integral annihilator of $c^{(i)}$ for every $i \in \mathcal{I}$, then the configurations in the collection have a common dilation constant with respect to $f$.
\end{lemma}

\begin{proof}
Let $c_{\text{max}}$ be the maximum absolute value of the coefficients of the configurations $c^{(i)}$.
Since the configurations $c^{(i)}$ are over the same alphabet, such number exists.
Let $f = \sum_{\bv \in \supp(f)} f_{\bv}X^{\bv}$ and define $s= c_{\text{max}} \sum_{\bv \in \supp(f)} |f_{\bv}|$.
In the proof of the dilation lemma in \cite{fullproofs}, it was shown that $r=s!$ is a dilation constant of any $c^{(i)}$.
The claim follows.
\end{proof}

\noindent
We need one more lemma from \cite{fullproofs}.

\begin{lemma}[\cite{fullproofs}]
\label{lemma3}
Let $c$ be a configuration. If $\varphi_1,\ldots,\varphi_m$ are line polynomials such that $\varphi_1^{e_1}\cdots\varphi_m^{e_m}$ annihilates $c$ for some positive integers $e_1,\ldots,e_m$,
then also $\varphi_1\cdots\varphi_m$ annihilates $c$.
\end{lemma}

\noindent
Combining these results, we obtain a structured annihilator common to a family of configurations.

\begin{theorem}
\label{theorem: special annihilator common}
Let $\mathcal{I}$ be an arbitrary index set, and let $\left ( c^{(i)} \right ) _{i \in \mathcal{I}}$ be a collection of configurations over the same alphabet $\A \subseteq \Z$.
  If $f$ is a non-trivial annihilator of $c^{(i)}$ for every $i \in \mathcal{I}$, then the configurations in the collection have a common annihilator which is of the form 
  $$
  (X^{\bv_1}-1)\cdots(X^{\bv_m}-1)
  $$
  where the vectors $\bv_1,\ldots,\bv_m$ are pairwise non-parallel.
\end{theorem}

\begin{proof}
By Lemma \ref{lemma2}, the configurations $c^{(i)}$ have a common dilation constant $r$ with respect to $f$ and hence by Lemma \ref{lemma1} they have a common annihilator
$$
g=\prod_{\bv \in \supp(f) \setminus \{\bu\}} (X^{r(\bv-\bu)}-1) =(X^{\bw_1}-1)\cdots (X^{\bw_k}-1).
$$
Assume that for some $i\neq j$, the vectors $\bw_i$ and $\bw_j$ are parallel, that is,
$q\bw_i= p \bw_j$ for some $p,q \in \Z\setminus\{0\}$.
  Since 
  $$
  X^{p \bw_j}-1 = (X^{\bw_j}-1)(X^{(p-1)\bw_j} + \ldots + X^{\bv_j} + 1),
  $$
   we may replace $X^{p\bw_j}-1$ in $g$ by its multiple $X^{p\bw_j}-1$. Similarly, 
   $X^{\bw_i}-1$ divides $X^{q \bw_i}-1=X^{p \bw_j}-1$, and so we may replace
   also $X^{\bw_i}-1$ in $g$ by $X^{p\bw_j}-1$. Thus we replaced
 $(X^{p\bw_j}-1)(X^{\bw_j}-1)$ by $(X^{p\bw_j}-1)^2$.
  Iterating this argument we conclude that the configurations $c^{(i)}$ have a common annihilator
  $$
  (X^{\bv_1}-1)^{e_1} \cdots (X^{\bv_m}-1)^{e_m}
  $$
  where the vectors $\bv_1,\ldots,\bv_m$ are pairwise linearly independent over $\Q$ and $e_1,\ldots,e_m$ are positive integers.
  Consequently, by Lemma \ref{lemma3}, we conclude that the configurations $c^{(i)}$ have a common annihilator
  $$
    (X^{\bv_1}-1)\cdots(X^{\bv_m}-1)
  $$
where the vectors $\bv_1,\ldots,\bv_m$ are pairwise non-parallel.  
\end{proof}

\subsubsection*{Periodic decomposition}


As a consequence of Theorem \ref{special annihilator configurations}, it follows that configurations admitting annihilators can be decomposed into finitely many periodic components.

\begin{theorem}[Periodic decomposition theorem \cite{fullproofs}]\label{thm: regular decomposition theorem}
  Let $c \in \A^{\Z^d}$ be an integral configuration with a non-trivial integral annihilator $f$.
  Then $c=c_1+\ldots+c_m$ where $c_1,\ldots,c_m \in \Z^{\Z^d}$ are periodic in pairwise linearly independent directions.
\end{theorem}

\begin{remark}\label{remark: periodic decomposition theorem}
The periodic functions $c_i$ in the periodic decomposition of a configuration $c$ may not be configurations, that is, their ranges may be infinite \cite{kari-szabados-multidimensional_words}.
\end{remark}

\section{Delone sets} \label{sec: delone sets}



In this section, we introduce Delone sets and related notions of local complexity, which provide the geometric setting for our results. 
These concepts allow us to extend the algebraic framework of annihilators from $\mathbb{Z}^d$-configurations to configurations supported on subsets of $\mathbb{R}^d$.
Our considerations are mostly adapted from \cite{aperiodic1, part1, lagarias_pleasants_2, lagarias_pleasants_ETDS, meyer}.

Let $d$ be a positive integer.
A subset $S \subseteq \R^d$ of the $d$-dimensional Euclidean space is \emph{uniformly discrete} if there exists a positive real number $r$ such that any open Euclidean ball of radius $r$ in $\R^d$ contains at most one point of $S$, and it is \emph{relatively dense} if there exists a positive real number $R$ such that any closed Euclidean ball of radius $R$ in $\R^d$ contains at least one point of $S$.
A subset of $\R^d$ is \emph{locally finite} if its intersection with any bounded set is finite.
Clearly, any uniformly discrete set is also locally finite.

\begin{definition}
A set $S \subseteq \R^d$ is a ($d$-dimensional) \emph{Delone set} if it is both uniformly discrete and relatively dense.
The largest possible uniform discreteness constant $r$ of a Delone set $S$ is called the \emph{packing radius} of $S$ and the smallest possible relative denseness constant $R$ of $S$ is called the \emph{covering radius} of $S$.
\end{definition}

\subsection{The patch complexity of Delone sets}

To quantify the local structure of Delone sets, we use patch complexity, which serves as the geometric analogue of pattern complexity in symbolic dynamics.


\begin{definition}
Let $S \subseteq \R^d$ be a Delone set, and let $T \geq 0$ be a real number.
The \emph{$T$-patch} of $S$ centered at $\vec{s} \in S$ is the set 
$$
\mathcal{P}_S(\bs,T) = S \cap B_T(\vec{s}).
$$
\end{definition}

\noindent
We say that two $T$-patches $\mathcal{P}_S(\bs_1,T)$ and $\mathcal{P}_S(\bs_2,T)$ of $S$ are (translation) \emph{equivalent} if $\mathcal{P}_S(\bs_1,T) - \bs_1 = \mathcal{P}_S(\bs_2,T) - \bs_2$ and denote $\mathcal{P}_S(\bs_1,T) \sim \mathcal{P}_S(\bs_2,T)$.
Otherwise, we say that $\mathcal{P}_S(\bs_1,T)$ and $\mathcal{P}_S(\bs_2,T)$ are \emph{inequivalent} and denote $\mathcal{P}_S(\bs_1,T) \not \sim \mathcal{P}_S(\bs_2,T)$.
Clearly, any $T$-patch of a Delone set is a finite set since Delone sets are uniformly discrete and hence locally finite.

\begin{definition}
The \emph{patch-counting function} $N_S(T)$ of a Delone set $S \subseteq \R^d$ gives for any $T \geq 0$ the number of distinct $T$-patches of $S$ up to translation equivalence, that is,
$$
N_S(T) = | \{ (S \cap B_T(\vec{s})) - \vec{s} \mid \vec{s} \in S \}| = |\{ (S-\bs) \cap B_T \mid \bs \in S \}|.
$$
\end{definition}

\noindent
In general, $N_S(T)$ may be infinite for sufficiently large $T$.

It is known that sufficiently small patch complexity of a Delone set implies its strong periodicity as the following theorem shows.

\begin{theorem}[\cite{lagarias_pleasants_2}] \label{thm:forced periodicity Lagarias}
  Let $S$ be a Delone set with covering radius $R$.
  If 
  $$
  N_S(T) < \frac{T}{2R}
  $$
  for some $T>0$, then $S$ is strongly periodic, that is, it has $d$ linearly independent periods.
\end{theorem}

\begin{remark}\label{remark:ideal crystals}
  In the literature, and in the statement of Theorem \ref{thm:forced periodicity Lagarias}, strongly periodic Delone sets are called \emph{ideal crystals}.
  The standard definition of ideal crystals looks quite different from our definition of strongly periodic Delone sets.
  Indeed, in \cite{lagarias_pleasants_2} a Delone set $S \subseteq \R^d$ is defined to be an ideal crystal if it has a full rank lattice of translational symmetries, that is, if there exists a non-empty finite set $F \Subset \R^d$ such that $S = F + \Lambda_S$ where $\Lambda_S = \{\bt \in \R^d \mid S + \bt = S \}$ is the set of translational symmetries of $S$.
  However, the definitions are equivalent:
  If a Delone set $S$ is strongly periodic, then it has $d$ linearly independent periods $\bv_1,\ldots,\bv_d\in\Lambda_S$ over $\R$.
  Since $S$ is a Delone set, the set
  $$
  F = S \cap \{ a_1 \bv_1+ \ldots + a_d \bv_d \mid a_1,\ldots,a_d \in [0,1) \}
  $$
  is finite and non-empty.
  It follows that $S = F + \Lambda_S$ and hence $S$ is an ideal crystal.
  Conversely, assume that $S$ is an ideal crystal, that is, $S = F + \Lambda_S$ for some $F \Subset \R^d$.
  If $S$ is not strongly periodic, then $\Lambda_S$ is contained in some subspace $V$ of $\R^d$ of dimension at most $d-1$.
  This implies that $S= F + \Lambda_S$ is not relatively dense. A contradiction.
\end{remark}

\subsection{Classes of Delone sets}

We now introduce several important classes of Delone sets that capture different levels of geometric regularity.
There are three well-known classes of Delone sets.
Only two of them are discussed in this paper. However, we will still briefly introduce the third class.

\begin{definition}
A Delone set $S$ is a \emph{Delone set of finite local complexity} (FLC) if $N_S(T)$ is finite for every $T$.
\end{definition}

\noindent
Sometimes Delone sets of finite local complexity are also called \emph{Delone sets of finite type} since they have only finitely many distinct ``local neighborhoods''. Equivalently, a Delone set $S$ is a Delone set of finite local complexity iff the set $S-S$ is locally finite.
In fact, already if the set $(S-S) \cap B_{2R}$ is finite where $R$ is a relative denseness constant of $S$, then $S$ is a Delone set of finite local complexity \cite{part1}.

\begin{definition}\label{def: meyer sets}
A Delone set $S$ is a
\emph{Meyer set} if also $S-S$ is a Delone set.
\end{definition}


\noindent
Clearly, any Meyer set is also a Delone set of finite local complexity.
The class of
Meyer sets was originally introduced by Meyer in \cite{meyer} under the name ``quasicrystal'' as the set of Delone sets $S$ such that $S-S \subseteq S+F$ for some non-empty finite set $F \Subset \R^d$.
However, this definition is equivalent to Definition \ref{def: meyer sets} \cite{Meyer-Lagarias}.
Sometimes Meyer sets are called also \emph{almost lattices} \cite{aperiodic1}.
It is quite easily seen that all strongly periodic Delone sets are Meyer sets.
This is usually considered as folklore.
For completeness, let us provide a proof for this fact.

\begin{lemma}\label{lemma:strongly periodic delone sets are meyer}
  If a Delone set $S \subseteq \R^d$ is strongly periodic, then it is a Meyer set.
\end{lemma}

\begin{proof}
  As noted in Remark \ref{remark:ideal crystals} strongly periodic Delone sets are ideal crystals, that is, there exists a non-empty finite set $F \Subset \R^d$ such that $S = F + \Lambda_S$ where
  $\Lambda_S = \{\bt \in \R^d \mid S + \bt = S \}$.
  Then
  \begin{align*}
    S-S &= F + \Lambda_S - (F + \Lambda_S)\\
    & =F-F + \Lambda_S - \Lambda_S \\
    & =F-F + \Lambda_S \\
    &=F + \Lambda_S - F \\
    &=  S + (-F).
  \end{align*}
  Thus, $S$ is a Meyer set by the original definition of Meyer sets.
\end{proof}

\noindent
The statement in the above lemma is not an equivalence since there are also non-periodic Meyer sets (for example, the set $\Z \setminus \{0\}\subseteq\R$).

Let us denote by $[S] = \Z[S] = \{ a_1 \vec{s}_1 + \ldots + a_k \vec{s}_k \mid \vec{s}_i \in S, a_i \in \Z, k \in \N \}$ the abelian group generated by $S \subseteq \R^d$.

\begin{definition}
A Delone set $S$ is a \emph{finitely generated Delone set} if the set $[S-S]$ or equivalently the set $[S]$ is a finitely generated abelian group of $\R^d$.
\end{definition}

\noindent
The class of finitely generated Delone sets contains the class of Delone sets of finite local complexity, that is, any Delone set of finite local complexity is a finitely generated Delone set \cite{part1}.

In the following example we see that the mentioned inclusions of different classes of Delone sets are strict.

\begin{example}\label{ex: strict hierarchy}
  We give 1-dimensional examples of 
  a Delone set which is not a finitely generated Delone set, a finitely generated Delone set which is not a Delone set of finite local complexity, and a Delone set of finite local complexity which is not a Meyer set.
  See Figure \ref{fig: hierarchy examples} for pictorial illustrations.
  In the following, we use the fact that any subgroup of a finitely generated abelian group is also a finitely generated abelian group.
  This follows from the fundamental theorem of finitely generated abelian groups \cite{abstract-algebra}.
  \begin{itemize}
    \item Consider the Delone set
    $$
    S_1 = \{ n + \frac{1}{n} \mid n \in \Z \setminus \{ 0 \} \}.
    $$
    It can be quite easily verified that $\Q \subseteq [S_1]$. Since $\Q$ is not a finitely generated abelian group, neither is its superset $[S_1]$.
    So, $S_1$ is not a finitely generated Delone set.
    
    \item Consider the Delone set
    $$
    S_2 = \{ n \pi \mid n \in \Z \} \cup \Z \setminus \{ \lfloor n \pi \rfloor , \lceil n \pi \rceil \mid n \in \Z \}.
    $$
    We have $[S_2-S_2] \subseteq \Z[1, \pi]$ which means that $S_2$ is a finitely generated Delone set since any subgroup of a finitely generated abelian group is finitely generated.
    However, the set of gaps between consecutive points of $S_2$ is infinite which implies that $S-S$ is not locally finite.
    Hence, $S_2$ is not a Delone set of finite local complexity.
    So, $S_2$ is a finitely generated Delone set but not a Delone set of finite local complexity.
    
    \item Consider the Delone set
    $$
    S_3 = - \N \cup \{ n \pi \mid n \in \N \}.
    $$
    Since there are only two kinds of gaps between two consecutive points of $S_3$, it is a Delone set of finite local complexity.
    However, we have $\N \subseteq S_3-S_3$ and $\{ n \pi \mid n \in \N \} \subseteq S_3 - S_3$.
    The latter set does not contain any integers, but it contains real numbers that are arbitrarily close to positive integers by the irrationality of $\pi$.
    This means that $S_3-S_3$ is not uniformly discrete and hence not a Delone set implying that $S_3$ is not a Meyer set.
  \end{itemize}
\end{example}

\begin{figure}[ht]
\centering

\begin{tikzpicture}[scale=0.5]
  \foreach \i in {1,...,10}{
    \draw[red,fill=red] (\i+1/\i,0) circle(2pt);
  }
    
  \foreach \i in {-10,...,-1}{
    \draw[red,fill=red] (\i+1/\i,0) circle(2pt);
  }
    
  \foreach \i in {-10,...,10}{
    \draw (\i,0.2) -- (\i,-0.2);
  }
    
  \draw[->] (-11,0) -- (11,0);
  
  \node at (-12.5,0) {$S_1:$};
  
  
  \draw[red,fill=red] (1,-2) circle(2pt);
  \draw[red,fill=red] (2,-2) circle(2pt);
  \draw[red,fill=red] (5,-2) circle(2pt);
  \draw[red,fill=red] (8,-2) circle(2pt);
  
  \draw[red,fill=red] (-1,-2) circle(2pt);
  \draw[red,fill=red] (-2,-2) circle(2pt);
  \draw[red,fill=red] (-5,-2) circle(2pt);
  \draw[red,fill=red] (-8,-2) circle(2pt);
  
  \draw[red,fill=red] (3.14159,-2) circle(2pt);
  \draw[red,fill=red] (2*3.14159,-2) circle(2pt);
  \draw[red,fill=red] (3*3.14159,-2) circle(2pt);
  
  \draw[red,fill=red] (-3.14159,-2) circle(2pt);
  \draw[red,fill=red] (-2*3.14159,-2) circle(2pt);
  \draw[red,fill=red] (-3*3.14159,-2) circle(2pt);
  
  \foreach \i in {-10,...,10}{
    \draw (\i,-1.8) -- (\i,-2.2);
  }
  \draw[->] (-11,-2) -- (11,-2);
  \node at (-12.5,-2) {$S_2:$};
  
  
  \draw[red,fill=red] (0,-4) circle(2pt);
  \draw[red,fill=red] (-1,-4) circle(2pt);
  \draw[red,fill=red] (-2,-4) circle(2pt);
  \draw[red,fill=red] (-3,-4) circle(2pt);
  \draw[red,fill=red] (-4,-4) circle(2pt);
  \draw[red,fill=red] (-5,-4) circle(2pt);
  \draw[red,fill=red] (-6,-4) circle(2pt);
  \draw[red,fill=red] (-7,-4) circle(2pt);
  \draw[red,fill=red] (-8,-4) circle(2pt);
  \draw[red,fill=red] (-9,-4) circle(2pt);
  \draw[red,fill=red] (-10,-4) circle(2pt);

  \draw[red,fill=red] (3.14159,-4) circle(2pt);
  \draw[red,fill=red] (2*3.14159,-4) circle(2pt);
  \draw[red,fill=red] (3*3.14159,-4) circle(2pt);
  
  \foreach \i in {-10,...,10}{
    \draw (\i,-3.8) -- (\i,-4.2);
  }
  \draw[->] (-11,-4) -- (11,-4);
  \node at (-12.5,-4) {$S_3:$};

  \end{tikzpicture}


\caption{The Delone sets defined in Example \ref{ex: strict hierarchy}.}
\label{fig: hierarchy examples}
\end{figure}
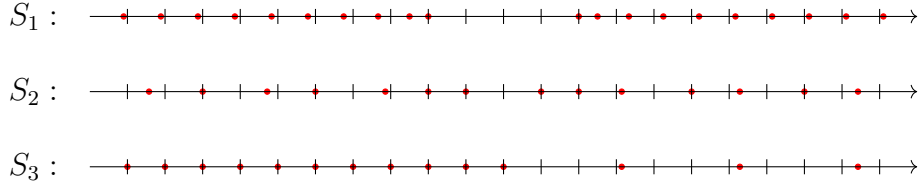

The following lemma states that the family of Delone sets of finite local complexity is closed under addition with non-empty finite sets.
We will use this lemma later in our considerations.

\begin{lemma} \label{closed under addition with a finite set}
  Let $S \subseteq \R^d$ be a Delone set of finite local complexity and let $F \Subset \R^d$ be a non-empty finite 
  set. 
  Then also $S+F$ is a Delone set of finite local complexity. 
\end{lemma}

\begin{proof}
  Clearly, $S + F$ is relatively dense since it contains a translation of $S$ as a subset. 
  Let us show that it is uniformly discrete too.
  For this, let $\vec{e}_1, \vec{e}_2 \in S+F$ be arbitrary but different. So, we have
  $\vec{e}_1 = \vec{s}_1 + \vec{f}_1$ and $\vec{e}_2 = \vec{s}_2 + \vec{f}_2$ for some $\vec{s}_1,\vec{s}_2 \in S$ and $\vec{f}_1,\vec{f}_2 \in F$
  and hence
  \begin{align*}
      ||\vec{e}_1 - \vec{e}_2|| 
      = ||\vec{s}_1 + \vec{f}_1 - (\vec{s}_2 + \vec{f}_2)|| 
      = 
      ||(\vec{s}_1 - \vec{s}_2) - (\vec{f}_2 - \vec{f}_1)||.
  \end{align*}
  Since $\vec{e}_1 \neq \vec{e}_2$, we also have $\vec{s}_1-\vec{s}_2 \neq \vec{f}_2-\vec{f}_1$ and since $F$ is finite, $\vec{f}_2-\vec{f}_1$ belongs to a finite set $F-F$. 
  By the fact that $S-S$ is locally finite,
  $||\vec{s}_1 + \vec{f}_1 - (\vec{s}_2 + \vec{f}_2)||$
  has a positive lower bound which is also a lower bound of 
  $||\vec{e}_1 - \vec{e}_2||$
  by the above computation.
  Thus, $S+F$ is uniformly discrete and hence a Delone set. 
  
  Finally, let us show that
  $S+F - (S+F)$ is locally finite, that is, $S+F$ is a Delone set of finite local complexity. For this it is enough to show that $\big ( S+F - (S+F) \big ) \cap B_{\varepsilon}(\vec{0})$ is finite for arbitrary $\varepsilon > 0$. This fact follows from computations
  $$
  \left ( S+F-(S+F) \right ) \cap B_{\varepsilon}(\vec{0}) \subseteq \left ( ( S-S  ) \cap B_{\varepsilon + T}(\vec{0}) \right ) + F-F
  $$
  and from the assumption that $S$ is locally finite where $T = \max \{ ||a||  \mid   a \in F - F \}$.
  Thus, the set $S+F$ is a Delone set of finite local complexity too.
\end{proof}

\begin{remark}
Note that the above lemma does not hold for arbitrary Delone sets, that is, the sum of a Delone set and a finite set is not necessarily a Delone set. Indeed, consider the Delone set
$$
S=S_1 =  \{ n + \frac{1}{n}  \mid n \in \Z \setminus \{0\}  \} 
$$
defined in Example \ref{ex: strict hierarchy}
and let
$F=\{0,1\}$. 
We have $n + \frac{1}{n} \in S+ F$ and $n + \frac{1}{n-1} \in S+F$ for arbitrarily large $n$. Thus, $S+F$ is not uniformly discrete and hence not a Delone set.
(In fact, we could show that the lemma does not hold even for finitely generated Delone sets considering the finitely generated Delone set 
$S = S_2=\{ n \pi \mid n \in \Z \} \cup \Z \setminus \{ \lfloor n \pi \rfloor , \lceil n \pi \rceil \mid n \in \Z \}$
and the finite set $F=\{0,1\}$ as above.)
\end{remark}

\subsection{Delone configurations and some algebraic concepts}

To apply the algebraic methods introduced in Section \ref{sec: symbolic dynamics}, we represent Delone sets as configurations and extend the notions of annihilators and periodizers to this setting.

\begin{definition}
We call an $\R^d$-configuration $c \in \C^{\R^d}$ (that is, a finitary function) a \emph{Delone configuration}, a \emph{finitely generated Delone configuration}, a \emph{Delone configuration of finite local complexity} or a \emph{Meyer configuration} if its support $\supp(c)$ is 
a Delone set, a finitely generated Delone set, a Delone set of finite local complexity or a Meyer set, respectively.
\end{definition}

\noindent
Naturally, we extend our terminology of annihilators and periodizers to this setting. (Recall, that by an $\R^d$-polynomial we mean a finitely supported function on $\R^d$.)
\begin{definition}
An $\R^d$-polynomial $f \in \C^{\R^d}$ \emph{annihilates} (or is an annihilator of) a function $c \in \C^{\R^d}$ if $f * c = 0$, and it \emph{periodizes} (or is a periodizer of) $c$ if $f * c$ is strongly periodic
\end{definition}

\noindent
Consequently, $f$ annihilates or periodizes a set $S \subseteq \R^d$ if it annihilates or periodizes, respectively, its indicator function $\mathbbm{1}_S$.
From now on, we denote $fc=f*c$.

Also, we generalize the definition of line polynomials: 
\begin{definition}
An $\R^d$-polynomial $f$ is a \emph{line $\R^d$-polynomial} if its support $\supp(f)$ contains at least two points, and $\supp(f)$ is contained in a line, that is, $\supp(f) \subseteq \bu + \R \bv$ for some $\bu,\bv \in \R^d$.
\end{definition}

\noindent
Again, the vector $\bv$ is called a direction of $f$.
Analogously, we generalize the definition of difference polynomials and say that an $\R^d$-polynomial is a \emph{difference $\R^d$-polynomial} if it is of the form $X^{\bv}-1$.

As in the case of $\Z^d$-configurations, an $\R^d$-configuration is periodic if and only if it is annihilated by a non-trivial difference $\R^d$-polynomial.
Note that annihilation of an Delone-configuration by a line $\R^d$-polynomial does not necessarily imply periodicity as we will see in Example \ref{ex:non-periodic with gen line pol ann}.
This differs from $\Z^d$-configurations for which annihilation by a line polynomial implies periodicity.
However, if the $\R^d$-configuration in consideration is a Delone configuration of finite local complexity, then annihilation by a line $\R^d$-polynomial indeed implies periodicity as we will prove in Lemma \ref{line polynomial annihilator implies periodicity}.

\subsection{An example}

We conclude the section with an example illustrating how non-trivial annihilators can arise naturally in geometric settings. 
The construction is based on rotations on the torus and produces Delone sets with non-trivial algebraic structure.


The (2-dimensional) \emph{torus} is the set $\mathbbm{T}^2 =[0,1) \times [0,1)$ equipped with the metric $d$ defined such that 
$$
d((u_1,u_2),(v_1,v_2)) = \max \{ d_1(u_1,v_1), d_1(u_2,v_2) \}
$$
for all $(u_1,u_2),(v_1,v_2) \in \mathbbm{T}^2$
where $d_1(u,v) = \min\{ |u-v|,1-|u-v| \}$ for all $u,v \in [0,1)$. In fact, $d_1$ is the metric in the 1-dimensional torus, that is, the circle $[0,1)$.

A \emph{rotation} $\rho_{a,b} \colon \mathbbm{T}^2 \to \mathbbm{T}^2$ of the torus by the parameters $a,b\in\R$ is defined such that
$$
\rho_{a,b}(u_1,u_2) = (\{ u_1+a \},\{ u_2+b \})
$$
for all $(u_1,u_2)\in \mathbbm{T}^2$.
For a reference on torus rotations, see \emph{e.g.} Section 1.11 in \cite{kurka}.

The torus $\mathbbm{T}^2$ is partitioned into sets $$A_0 = \{ (u_1,u_2) \in \mathbbm{T}^2 \mid 0 \leq u_1+u_2 < 1 \}$$ and $$A_1 = \{ (u_1,u_2) \in  \mathbbm{T}^2 \mid 1 \leq  u_1+u_2 < 2 \}.$$

For a vector $\bz=(z_1,z_2) \in \mathbbm{T}^2$ and a non-zero real number $\alpha$, define a binary $\Z^2$-configuration $c_{\bz,\alpha} \in \{0,1\}^{\Z^2}$ such that
$c_{\bz,\alpha}(i,j) = 1$ if $\rho_{i\alpha,j\alpha}(\bz) \in A_1$ and
$c_{\bz,\alpha}(i,j) = 0$ if $\rho_{i\alpha,j\alpha}(\bz) \in A_0$. 

\begin{remark} 
The configuration $c_{\bz,\alpha}$ encodes the orbit of $\bz$ under the $\mathbb{Z}^2$-action defined by $(i,j) \cdot \bz \mapsto c_{\bz,\alpha}$.
While different choices of $\bz$ lead to different configurations, they share the same underlying algebraic properties, including the existence of non-trivial annihilators as we will see in Lemma \ref{lemma:torus}.
\end{remark}

Define a set
$$
S_{\bz,\alpha} = \{ \bz + (i,j) \mid c_{\bz,\alpha}(i,j) = 1 \} = \bz + \supp(c_{\bz,\alpha}) \subseteq \R^2.
$$
As a subset of the uniformly discrete set $\bz + \Z^2$, the set
$S_{\bz,\alpha}$ is also uniformly discrete. 
Moreover, if $\alpha$ is
irrational, then the set
$$
D_{\alpha}=\{(\{i\alpha\},\{j\alpha\}) \mid i,j\in\mathbb Z\}
$$
is dense in $\mathbbm T^2$ (see Section 1.11 in \cite{kurka}). Since
$A_1$ has a non-empty interior, it has an open subset $U \subseteq A_1$.
The translates of $U$ form an open cover of $\mathbbm{T}^2$.
By density of $D_{\alpha}$ we may consider translates by elements of $D_{\alpha}$,
and by compactness
of $\mathbbm{T}^2$ already finitely many such translates form an open cover of $\mathbbm{T}^2$.
Consequently, there exists a finite set $F\subseteq\Z^2$ such that
for every $(i,j)\in\Z^2$ there is $(a,b)\in F$ with
$c_{\bz,\alpha}(i-a,j-b)=1$.
It follows that
$\supp(c_{\bz,\alpha})$ is relatively dense 
and hence $S_{\bz,\alpha}$ is relatively dense. 
Thus, $S_{\bz,\alpha}$ is a Delone set whenever
$\alpha$ is irrational.


We now show that the configurations $c_{\bz,\alpha}$ and hence the corresponding 
sets $S_{\bz,\alpha}$, admit non-trivial annihilators.
The proof of this fact goes by observing that the configurations $c_{\bz,\alpha}$ are sums of periodic functions.

\begin{lemma}\label{lemma:torus}
For any $\bz=(z_1,z_2) \in \mathbbm{T}^2$ and non-zero $\alpha \in \R$, we have
$$
c_{\bz,\alpha}(i,j) = \lfloor z_1+z_2+ (i+j) \alpha \rfloor - \lfloor z_1 + i \alpha \rfloor - \lfloor z_2 + j \alpha \rfloor
$$
for all $(i,j) \in \Z^2$.
Consequently, $c_{\bz,\alpha}$ is annihilated by the polynomial $$(X^{(1,-1)}-1)(X^{(0,1)}-1)(X^{(1,0)}-1).$$
\end{lemma}

\begin{proof}
The first claim follows from the following computation:
\begin{align*}
& \ \ \ \ \lfloor z_1+z_2+ (i+j) \alpha \rfloor - \lfloor z_1 + i \alpha \rfloor - \lfloor z_2 + j \alpha \rfloor  \\
&= \{z_1 + i \alpha\} + \{z_2 + j \alpha \} - \{z_1 + z_2 +(i+j) \alpha\} \\
&= 
\lfloor \{z_1 + i \alpha\} + \{z_2 + j \alpha \} \rfloor + \{ \{z_1 + i \alpha\} + \{z_2 + j \alpha \} \} - \{z_1 + z_2 +(i+j) \alpha\}
\\
&=
\lfloor \{z_1 + i \alpha\} + \{z_2 + j \alpha \} \rfloor \\
&=
\begin{cases}
    1, \text{ if } \{z_1 + i \alpha\} + \{z_2 + j \alpha \} \in A_1 \\
    0, \text{ if } \{z_1 + i \alpha\} + \{z_2 + j \alpha \} \in A_0
\end{cases} \\ 
&=
c_{\bz,\alpha}(i,j).
\end{align*}
So, we have $c_{\bz,\alpha} = c_1+c_2+c_3$ where $c_1(i,j) =\lfloor z_1+z_2+ (i+j) \alpha \rfloor$,
$c_2(i,j) =-\lfloor z_1+ i \alpha \rfloor$, and
$c_3(i,j) =-\lfloor z_2+ j \alpha \rfloor$
for all $(i,j)\in\Z^2$.
These functions are annihilated by the polynomials $X^{(1,-1)}-1$, $X^{(0,1)}-1$, and $X^{(1,0)}-1$, respectively.
Thus, $c_{\bz,\alpha}$ is annihilated by the product of these polynomials and hence the second claim follows.
\end{proof}

Thus, this construction provides examples of non-periodic Delone sets with non-trivial annihilators. 
It illustrates how algebraic relations between translates can arise from the underlying dynamical structures.
If the parameter $\alpha$ is irrational, then the corresponding configuration $c_{\bz,\alpha}$ is non-periodic but still has a non-trivial annihilator.


\begin{remark}
In \cite{icalp, kari-szabados-multidimensional_words} the authors gave the configuration $c_{\vec{0},\alpha}$ for an irrational $\alpha$ as an example of a configuration that has a non-trivial annihilator but cannot be expressed as a sum of periodic finitary functions. 
\end{remark}

\section{Delone configurations with annihilators} \label{sec: delone configurations with annihilators}

In this section, we develop an algebraic theory of Delone configurations with annihilators. We begin by extending results from the $\mathbb{Z}^d$-setting to $\mathbb{R}^d$, and then refine these results for configurations supported on Delone sets of finite local complexity. 
Finally, we show that configurations admitting annihilators can be decomposed to finitely many periodic components.

\subsection{General setting}

We begin by extending the connection between low complexity and annihilators from $\mathbb{Z}^d$-configurations to $\mathbb{R}^d$-configurations.
This is a direct generalization of Lemma \ref{th:low_complexity}.
For completeness, we provide a proof for the claim. 
The orthogonal complement of a linear subspace $U$ 
of an inner-product space is denoted by $U^{\bot}$. In the following, the inner-product is the scalar product of complex vectors. For a complex number $z\in\C$, we denote its complex conjugate by $\overline{z}$, as usual.

\begin{lemma} \label{lemma: non-trivial annihilator}
  Let $c \in \C^{\R^d}$ be an $\R^d$-configuration and let $D = \{ \bd_1, \ldots , \bd_m \} \Subset \R^d$ be an $\R^d$-shape.
  If $c$ has low complexity with respect to $D$, that is, if
  $P_c(D) \leq m$, then $c$ has a periodizer of the form 
  $$a_1 X^{-\bd_1} + \ldots + a_m X^{-\bd_m}$$
for some non-zero $(a_1,\ldots,a_m) \in \C^m$.
\end{lemma}

\begin{proof} 
  Consider the set
  $$
  V = \left \{ (1, c_{\vec{d}_1+\vec{v}}, \ldots, c_{\vec{d}_m+\vec{v}}) \mid \vec{v} \in \R^d \right \}.
  $$
  The subspace $L(V) \subseteq \C^{m+1}$ generated by $V$ has dimension at most $m$ since $|V| = P_c(D) \leq  m$.
  Thus, $L(V)^{\bot}$ is non-trivial. Let $ (\overline{a}_0, \ldots, \overline{a}_m) \in L(V)^{\bot} \setminus \{ \vec{0} \}$. Then
  $$
  a_0 \cdot 1 + a_1 \cdot c_{\vec{d}_1 + \vec{v}} + \ldots + a_m \cdot c_{\vec{d}_m+\vec{v}} = 0
  $$
  for all $\vec{v} \in \R^d$. 
  Thus, $a_1 X^{- \vec{d}_1} + \ldots + a_n X^{- \vec{d}_m}$ is a non-trivial periodizer of $c$.
\end{proof}

\begin{remark}
    The above lemma extends naturally to configurations over any abelian group $G$ with the same proof.
    In other words, any $G$-configuration $c \in \C^G$ with low complexity has a non-trivial periodizer.
    Pattern complexity, low complexity, periodizers, and annihilators for $G$-configurations over arbitrary abelian groups $G$ are defined analogously to those for $\Z^d$-configurations and $\R^d$-configurations.
    We omit the details.
    This means that the connection between low complexity and the existence of annihilators is purely algebraic, and it is not tied specifically to the lattice or the Euclidean setting.
\end{remark}

%

As noted earlier, annihilation by a line $\R^d$-polynomial does not necessarily imply periodicity.
However, if an $\R^d$-configuration has a line $\R^d$-polynomial annihilator of a particularly simple type, namely a power of a difference $\R^d$-polynomial, then it is periodic as the following lemma shows.

\begin{lemma}\label{lemma:difference polynomial power annihilator}
  Assume that a configuration $c \in \A^{\R^d}$ has an annihilator of the form $g(X)=(X^{\bv}-1)^k$ for some positive integer $k$ and $\bv \in \R^d \setminus \{\vec{0}\}$.
  Then $c$ is periodic in direction $\bv$.
  More precisely, $c$ is $p \bv$-periodic for some non-zero integer $p$.
\end{lemma}

\begin{proof}
  By the binomial theorem we have $g(X) = \sum_{j=0}^k a_j X^{j \bv}$ where $a_0,a_k \neq 0$.
  Let us define a $\Z$-configuration $e^{(\bu)} \in \A^{\Z}$ for any $\bu \in \R^d$ such that
  $$
  e^{(\bu)}_i = c_{\bu + i \bv}
  $$
  for all $i \in \Z$.
  It follows that $g' = \sum_{j=0}^k a_j x^{j}$ annihilates $e^{(\bu)}$ since
  $$
  (g'e^{(\bu)})_i = \sum_{j=0}^k a_j e^{(\bu)}_{i-j} = \sum_{j=0}^k a_j c_{\bu + (i-j) \bv} = (gc)_{\bu + i \bv} = 0
  $$
  for all $i \in \Z$.
  Hence, each $e^{(\bu)}$ is periodic.
  Moreover, the smallest periods of $e^{(\bu)}$ are bounded and hence $c$ is periodic in direction $\bv$.
  In fact, $c$ is $p \bv$-periodic where $p$ is a common period of the configurations $e^{(\bu)}$.
  The claim follows.
\end{proof}

\begin{remark}
  In the proof of the above lemma we in fact proved that if an $\R^d$-configuration $c$ is annihilated by a line $\R^d$-polynomial $f$ such that
  $$
  \supp(f) \subseteq \bu + \Z \bv
  $$
  for some $\bu,\bv \in \R^d$, then it is periodic.
\end{remark}

We now prove the main structural result of this section, which generalizes the corresponding result for $\mathbb{Z}^d$-configurations (Theorem \ref{special annihilator configurations}).

\begin{theorem} \label{main result 1}
  Let $c$ be an integral $\R^d$-configuration and assume that it has a non-trivial integral annihilator $f$. 
  Then for every $\bu \in \supp(f)$ it has an annihilator of the form
  $$
  (X^{\vec{v}_1} - 1) \cdots (X^{\vec{v}_m} - 1)
  $$
  where each $\bv_i$ is parallel to $\bu_i - \bu$ over $\Q$ for some $\bu_i \in \supp(f) \setminus \{ \bu  \}$.
  Moreover, the vectors $\bv_i$ can be chosen to be pairwise linearly independent over $\Q$.
\end{theorem}

\begin{proof}
  Assume that $f = \sum_{i=1}^n f_i X^{\bu_i}$ for some non-zero $f_i \in \Z$ where $\supp(f) = \{ \vec{u}_1, \ldots , \vec{u}_n \}$. 
  Define for any $\vec{u} \in \R^d$ a $\Z^n$-configuration 
  $c^{(\vec{u})} \in \A^{\Z^n}$
  over the same finite alphabet $\A$ as $c$ such that 
  $$
  c^{(\vec{u})}(i_1, \ldots, i_n) = c(\vec{u} + i_1 \vec{u}_1 + \ldots + i_n \vec{u}_n)
  $$
  for all 
  $(i_1,\ldots,i_n) \in \Z^n$. 
  
  The polynomial $g = 
  g(T) = g(t_1, \ldots, t_n) = \sum_{i=1}^n f_i t_i = \sum_{i=1}^n f_i T^{\be_i}$ is a non-trivial annihilator of $c^{(\vec{u})}$ for any $\vec{u} \in \R^d$ since 
  \begin{align*}
  (g c^{(\bu)})(i_1,\ldots,i_n) 
  &= f_1 c^{(\bu)}(i_1-1,i_2,\ldots,i_n) + \ldots + f_n c^{(\bu)}(i_1,\ldots,i_{n-1},i_n-1) \\
  &= f_1 c(\bu + (i_1-1) \bu_1 + i_2 \bu_2 + \ldots + i_n \bu_n) + \ldots \\
  & \ \ \ + f_n c(\bu + i_1 \bu_1 + \ldots + i_{n-1}\bu_{n-1} + (i_n - 1) \bu_n) \\
  &
  = (fc)(\bu + i_1 \bu_1 + \ldots + i_n \bu_n) \\
  &= 0
  \end{align*}
  for all $(i_1,\ldots,i_n) \in \Z^n$.
  
  By Lemma \ref{lemma2} the $\Z^n$-configurations $c^{(\bu)}$ have a common dilation constant with respect to $g$ and hence 
  by Theorem \ref{theorem: special annihilator common} they have for all $\vec{e} \in \supp(g) = \{ \vec{e}_1, \ldots , \vec{e}_n \}$ a common annihilator of the form
  $$
  (T^{\vec{w}_1}-1) \cdots (T^{\vec{w}_m}-1)
  $$
  where each $\vec{w}_i$ is parallel to $\vec{e}_j - \vec{e}$ for some $\vec{e}_j \in \supp(g) \setminus \{\vec{e}\}$, that is, $\vec{w}_i = k \vec{e}_j - k \vec{e}$ for some non-zero integer $k$.
  
  Let us denote $\vec{w}_i = (w_{i,1}, \ldots , w_{i,n})$ for $i \in \{ 1, \ldots , m \}$ and let $i_0 \in \{1,\ldots,n\}$ be such that $\be = \be_{i_0}$. 
  Note that each $\bw_i$ has exactly two non-zero entries, and they are equal to $k$ and $-k$.
  It follows that the $\R^d$-polynomial
  $$
  (X^{w_{1,1} \vec{u}_1 + \ldots + w_{1,n} \vec{u}_n} - 1) \cdots (X^{w_{m,1} \vec{u}_1 + \ldots + w_{m,n} \vec{u}_n} - 1)
  $$
  is an annihilator of $c$ such that each $w_{i,1} \vec{u}_1 + \ldots + w_{i,n} \vec{u}_n$ is parallel to $\bu_j - \bu_{i_0}$.
  This holds for all $i_0 \in \{1,\ldots,n\}$ and hence we have proved the first part of the claim.

  For the ``moreover'' part, assume 
  that
  $$
  g = (X^{\vec{v}_1} - 1) \cdots (X^{\vec{v}_m} - 1)
  $$ annihilates $c$,
  and assume that for some distinct $i,j \in \{1,\ldots,m\}$ the vectors $\bv_i$ and $\bv_j$ are parallel over $\Q$, that is, $q\bv_i= p \bv_j$ for some $p,q \in \Z\setminus\{0\}$.
  We may replace $X^{\bv_i}-1$ in $g$ by $X^{p \bv_j}-1 = X^{q \bv_i}-1$.
  Since 
  $$
  X^{p \bv_j}-1 = (X^{\bv_j}-1)(X^{(p-1)\bv_j} + \ldots + X^{\bv_j} + 1),
  $$
   we may replace $(X^{p\bv_j}-1)(X^{\bv_j}-1)$ by $(X^{p\bv_j}-1)^2$.
  Iterating this argument, we conclude that $c$ has an annihilator
  $$
  (X^{\bw_1}-1)^{e_1} \cdots (X^{\bw_k}-1)^{e_k}
  $$
  where the vectors $\bw_1,\ldots,\bw_k$ are pairwise linearly independent over $\Q$ and $e_1,\ldots,e_k$ are positive integers.
  Consequently, the $\R^d$-configuration 
  $$(X^{\bw_2}-1)^{e_2} \cdots (X^{\bw_k}-1)^{e_k}c$$ 
  is annihilated by $(X^{\bw_1}-1)^{e_1}$.
  By Lemma \ref{lemma:difference polynomial power annihilator} it is periodic in direction $\bw_1$.
  Repeating this argument 
  we conclude that $c$ is annihilated by
  $$
  (X^{r_1 \bw_1}-1) \cdots (X^{r_k \bw_k}-1)
  $$
  for some non-zero $r_1,\ldots,r_k\in \Z$ and hence we have proved the claim.
\end{proof}



\noindent
The vectors $\bv_1,\ldots,\bv_m$ in the above theorem may not be pairwise linearly independent over $\R$ as the following example shows.
However, in the following subsection we see that if $c$ is a Delone configuration of finite local complexity, then they can be chosen to be pairwise linearly independent over $\R$.
The following example gives also an example of a non-periodic $\R$-configuration with a non-trivial annihilator.

The following example illustrates the fact that without additional geometric assumptions, one cannot expect the directions of the difference $\R^d$-polynomials in the above theorem to be linearly independent over $\mathbb{R}$.

\begin{example}\label{example:1D_annihilator_non-periodic}
  Consider the binary $\R$-configurations $c_1,c_2 \in \{0,1\}^{\R}$ defined such that
  $$
  c_1(i) = 
  \begin{cases}
    1 &\text{, if } i \in \Z \\
    0 &\text{, otherwise}
  \end{cases}
  $$
  and
  $$
  c_2(i) = 
  \begin{cases}
    1 &\text{, if } i \in \alpha \Z  \\
    0 &\text{, otherwise}
  \end{cases}
  $$
  where $\alpha$ is an irrational number.
  Both $c_1$ and $c_2$ are periodic, that is, annihilated by non-trivial difference $\R$-polynomials. 
  More precisely, $c_1$ is annihilated by $x - 1$ and $c_2$ is annihilated by $x^{\alpha} - 1$. 
  Consequently, their sum $c = c_1 + c_2$ has a non-trivial annihilator 
  $(x - 1)(x^{\alpha} - 1)$ but $c$ is non-periodic since $\alpha$ is irrational.
  If in the above theorem $\bv_1,\ldots,\bv_m$ could be chosen to be pairwise linearly independent over $\R$, this would mean in the 1-dimensional setting that $m=1$ and hence $c$ would be necessarily periodic.
  Note that the support of $c$ is not a Delone set.
\end{example}

In the proof of Theorem~\ref{main result 1} we noted that 
the extracted $\Z^n$-configurations $c^{(\bu)}$ have a common dilation constant with respect to $g$.
 It directly follows that the original $\R^d$-configuration $c$ has this same dilation constant with respect to $f$, yielding the following dilation lemma for $\R^d$-configurations.

 \begin{lemma}\label{lem:Rdilationlemma}
 Let $c$ be an integral $\R^d$-configuration and assume that it has a non-trivial integral annihilator $f$.
  There exists a positive integer $r$ such that for every positive integer $k$ with $\gcd(k,r)=1$ also $f(X^k)$ annihilates $c$.
\end{lemma}

\subsection{Delone configurations of finite local complexity}
\label{sec:Delone configurations of finite local complexity with annihilators}

We now specialize to Delone configurations of finite local complexity. In this setting, stronger results can be obtained.

We have seen that any integral $\R^d$-configuration with non-trivial annihilators has an annihilator which is a product of difference $\R^d$-polynomials.
Moreover, the directions of these difference $\R^d$-polynomials can be chosen to be pairwise linearly independent over $\Q$.
However, in Example \ref{example:1D_annihilator_non-periodic} we noticed that the directions of these difference $\R^d$-polynomials cannot always be chosen to be pairwise linearly independent over $\R$.

In this section we study Delone configurations of finite local complexity and improve Theorem \ref{main result 1}.
We show that if the $\R^d$-configuration in consideration is a Delone configuration of finite local complexity, then the directions of the difference $\R^d$-polynomial factors of the special annihilator in fact can be chosen to be pairwise linearly independent over $\R$ instead of just $\Q$.

We begin with some observations concerning Delone configurations of finite local complexity with non-trivial annihilators.
The following lemma generalizes the fact that any $\Z$-configuration with a non-trivial annihilator is necessarily periodic.

\begin{lemma} \label{DCFLC}
  Let $c \in \A^{\R}$ be a 1-dimensional Delone configuration of finite local complexity and let $S$ be the support of $c$ with uniform discreteness constant $r$ and relative denseness constant $R$.
  Assume that $c$ has a non-trivial annihilator $f$. 
  Then $c$ is periodic.
  Moreover, a period of $c$ belongs to the finite set $(S-S) \cap B_{2MR}$ where 
  $M \leq |\A|^{|(S-S) \cap B_{2R}|^{\lfloor\frac{k}{2r} \rfloor }}$ and $k$ is such that $\supp(f) \subseteq [0,k]$.
\end{lemma}

\begin{proof}
  Without loss of generality we may assume that $\supp(f) = \{ t_0, t_1, \ldots , t_n \}$ where $0 < t_1 < \ldots < t_n$ and $t_0=0$.
  So, we have
  $f = \sum_{i=0}^n f_i x^{t_i}$ for some non-zero $f_0, \ldots , f_n$. 
  Let $k$ be such that $\supp(f) \subseteq [0,k]$. 
  Since $S$ is 
  a Delone set, we can order it and write $S = \{ s_i \mid i \in \Z \}$ such that $s_i < s_{i+1}$ for every $i \in \Z$. 
  By the assumption that $S$ is a Delone set of finite local complexity, the set $(S-S) \cap B_{2R}$ is finite 
  and hence the set $\{ s_{i+1} - s_i \mid i \in \Z \} \subseteq (S-S) \cap B_{2R}$ of the lengths of intervals between two consecutive points of $S$ is finite. 
  This implies that for some $i < j$ 
  we have $\tau^{-s_i}(c) \restriction _{- [0,k]} = \tau^{-s_j}(c) \restriction _{- [0,k]}$, that is, $c_{s_i-u}=c_{s_j-u}$ for all $u\in [0,k]$.
  
  Since $f$ is an annihilator of $c$, we have
  $$
  f_0 c_{t} + f_1 c_{t-t_1} + \ldots + f_n c_{t-t_n} = 0
  $$
  for every $t \in \R$. 
  We can solve $c_t$ and $c_{t-t_n}$ from the equation because $f_0$ and $f_n$ are non-zero. 
  Thus, the content of $c$ in the interval $t - [0,k]$ determines whole $c$ since $\supp(f) \subseteq [0,k]$ 
  and
  hence the condition $\tau^{-s_i}(c) \restriction _{- [0,k]} = \tau^{-s_j}(c) \restriction _{- [0,k]}$ (or equivalently $\tau^{s_j-s_i}(c) \restriction _{- [0,k]}=c\restriction _{- [0,k]}$) which we obtained above
  implies that $c$ is $(s_j - s_i)$-periodic.
  
  Let $M = j -i$.
The set $s_i - [0,k]$ contains $s_i$ and at most $\lfloor \frac{k}{2r} \rfloor$ other points of $S$.
Moreover, there are at most $|(S-S)\cap B_{2R}|$ different gaps between consecutive points of $S$.
Together, these imply that there are at most $|(S-S)\cap B_{2R}|^{\lfloor \frac{k}{2r} \rfloor}$ different sets $\supp(\tau^{-s_i}(c)) \cap (-[0,k])$.
Thus, there are at most $|\A|^{|(S-S)\cap B_{2R}|^{\lfloor \frac{k}{2r} \rfloor}}$ different functions 
$\tau^{-s_i}(c) \restriction _{-[0,k]}$.
So, we may choose $M=j-i \leq |\A|^{|(S-S)\cap B_{2R}|^{\lfloor \frac{k}{2r} \rfloor}}$.
  This concludes the proof.
\end{proof}


Recall that in Example \ref{example:1D_annihilator_non-periodic} we showed that the above lemma does not hold
for arbitrary $\R$-configurations.
However, the support of the $\R$-configuration in the example is not even a Delone set.

\begin{remark}\label{remark: flc_subsets_1d_annihilation}
More generally, we can prove that any $\R$-configuration $c$ whose support is a subset of a Delone set of finite local complexity that has a non-trivial annihilator is necessarily periodic.
Indeed, if $c=0$, then it is trivially periodic, and if $c$ is non-zero, then the non-trivial annihilator defines a linear recurrence on $c$ and hence the support of $c$ is relatively dense.
This implies that it is a Delone set of finite local complexity since any relatively dense subset of a Delone set of finite local complexity is a Delone set of finite local complexity too.
The periodicity of $c$ follows from Lemma \ref{DCFLC}.
\end{remark}

%

Using Lemma \ref{DCFLC} we show that if a Delone configuration of finite local complexity has a line $\R^d$-polynomial annihilator, then it is periodic.
In the proofs of the following results for a non-zero vector $\bv \in \R^d$, by a $\bv$-fiber of an $\R^d$-configuration $c \in \A^{\R^d}$ we mean an $\R$-configuration $e^{(\bu)} \in \A^{\R}$ for some $\bu \in \R^d$ defined such that 
$$
e^{(\bu)}_r = c_{\bu + r \bv}
$$
for all $r \in \R$.

We next show that, in contrast to the general $\mathbb{R}^d$-setting, for Delone configurations of finite local complexity, annihilation by a line polynomial implies periodicity.

\begin{lemma} \label{line polynomial annihilator implies periodicity}
  Let $c \in \A^{\R^d}$ be a Delone configuration of finite local complexity and assume that it is annihilated by a line $\R^d$-polynomial $f$ in direction $\bv$. 
  Then $c$ is periodic in direction $\bv$.
\end{lemma}

\begin{proof}
    By multiplying $f$ with a suitable monomial we may assume that $f$ is of the form $f(X) = f_0 + f_1 X^{r_1 \vec{v}} + \ldots + f_n X^{r_n \vec{v}}$ where $0 < r_1 < \ldots < r_n$ with $n \geq 1$ and  $f_0,\ldots,f_n \neq 0$. 
    
    Consider an arbitrary $\bv$-fiber $e$ of $c$.
    Clearly, $e$ is annihilated by the polynomial 
    $f'(x) = f_0 + f_1 x^{r_1} + \ldots + f_n x^{r_n}$. 
    If $e=0$, then it is trivially periodic. 
    If $e \neq 0$, then $E = \supp(e)$ must be relatively dense since $f'$ defines a recurrence relation on $e$. 
    In fact, if $e \neq 0$, then $E$ is a Delone set of finite local complexity since $S=\supp(c)$ is a Delone set of finite local complexity.
    Thus, $e$ is a Delone configuration of finite local complexity
    and hence
    by Lemma \ref{DCFLC} it is periodic.

    It remains to show that all the $\bv$-fibers of $c$ have a common non-zero period.
    Then it follows that $c$ is periodic in direction $\bv$.
    Note that the supports of the non-zero fibers of $c$ have a common uniform discreteness constant and a common relative denseness constant. So, the number $M$ from Lemma \ref{DCFLC} can be chosen to be the same for every fiber $e$.
    Hence, there exist only finitely many possibilities for the smallest non-zero period (in absolute value) of $e$.
    Let $e_1$ and $e_2$ be two non-zero $\bv$-fibers of $c$ with smallest periods $p_1$ and $p_2$, respectively.
    There exist $\bu_1,\bu_2 \in S$ such that $\bu_1 + np_1\bv \in S$ and $\bu_2 + mp_2\bv \in S$ for all $n,m\in\Z$.
    Consequently, for any $\varepsilon >0$ there exist infinitely many $n$ and $m$ such that $np_1\bv$ and $mp_2\bv$ are within distance $\varepsilon$ from each other.
    Then $\bu_1 + np_1\bv$ and $\bu_2 + mp_2\bv$ are within distance $||\bu_1-\bu_2|| + \varepsilon$ from each other.
    We conclude that $p_1$ and $p_2$ must be rationally dependent. Otherwise, $(S-S) \cap B_{||\bu_1-\bu_2|| + \varepsilon}$ is an infinite set which is a contradiction with the fact that $S$ is a Delone set of finite local complexity, that is, $S-S$ is a locally finite set.
    So, the finitely many smallest periods of the $\bv$-fibers of $c$ are rationally dependent.
    It follows that the $\bv$-fibers of $c$ have a common period.
\end{proof}

\begin{remark}\label{remark: subset FLC line polynomial annihilator}
Similarly to Remark \ref{remark: flc_subsets_1d_annihilation} we can prove that the above lemma holds more generally for any $\R^d$-configuration whose support is a subset of a Delone set of finite local complexity.
Indeed, if the support of $c$ is a subset of a Delone set of finite local complexity, then the support of each $\bv$-fiber of $c$ is a subset of a Delone set of finite local complexity (actually, it is either a Delone set of finite local complexity or the empty set).
The claim follows by Remark \ref{remark: flc_subsets_1d_annihilation}.
\end{remark}

Lemma \ref{line polynomial annihilator implies periodicity} does not hold for general Delone configurations, that is, there are non-periodic Delone configurations with line $\R^d$-polynomial annihilators as the following example shows.

\begin{example}\label{ex:non-periodic with gen line pol ann}
  Let $\alpha$ be an irrational real number.
  Consider the Delone configuration $c \in \{0,1\}^{\R^2}$ defined such that
  $$
  c_{\bu} = 
  \begin{cases}
    1 &, \text{ if $\bu = (k, i)$ for any $k\in \Z$ and $i\in \Z\setminus \{0\}$} \\
    1 &, \text{ if $\bu = (k \alpha, 0)$ for any $k\in \Z$} \\
    0 &, \text{ otherwise}
  \end{cases}.
  $$
  It is non-periodic but has
  a line $\R^d$-polynomial annihilator $(x-1)(x^{\alpha}-1)$.
\end{example}

%
Next, we prove our improvement of Theorem \ref{main result 1}.


\begin{theorem} \label{main result 2}
  Let $c$ be an integral Delone configuration of finite local complexity and assume that it has a non-trivial integral annihilator $f$. 
  Then for every $\bu \in \supp(f)$ it has an annihilator of the form
  $$
  (X^{\vec{v}_1} - 1) \cdots (X^{\vec{v}_m} - 1)
  $$
  where the vectors $\bv_1,\ldots,\bv_m$ are pairwise linearly independent over $\R$ and each $\bv_i$ is parallel to $\bu_i - \bu$ over $\R$ for some $\bu_i \in \supp(f) \setminus \{ \bu  \}$.
\end{theorem}

\begin{proof}
  By Theorem \ref{main result 1} for every $\bu \in \supp(f)$ there exist vectors $\vec{v}_1,\ldots,\vec{v}_m$ such that
  $$
  g=(X^{\vec{v}_1} - 1) \cdots (X^{\vec{v}_m} - 1)
  $$
  annihilates $c$ where each $\bv_i$ is parallel to $\bu_i - \bu$ over $\Q$ for some $\bu_i \in \supp(f) \setminus \{ \bu  \}$.
  
  Consider the $\R^d$-polynomial $g$ for some fixed $\bu \in \supp(f)$.
  Let $\vec{w}_1 , \ldots , \vec{w}_k$ be pairwise linearly independent vectors over $\R$ such that each $\bv_i$ is parallel to some $\bw_j$ over $\R$.
  Moreover, assume that $k$ is minimal, that is, each $\bw_j$ is parallel to some $\bv_i$.
  
  Clearly, $k \leq m$. 
  If $k=m$, then all the difference $\R^d$-polynomials in the product are in pairwise linearly independent directions over $\R$, and we are done.
  So, we may assume that $k<m$. 
  Define $\R^d$-polynomials $\varphi_1,\ldots,\varphi_k$ such that $\varphi_i$ is the product of every $X^{\vec{v}_{j}}-1$ that has direction $\vec{w}_i$. 
  Note that every $\varphi_i$ is a line $\R^d$-polynomial since a product of line $\R^d$-polynomials in the same direction is still a line $\R^d$-polynomial in that same direction. 
  Clearly, $g = \varphi_1 \cdots \varphi_k$
  and hence $\varphi_1$ is a line $\R^d$-polynomial annihilator of $\varphi_2 \cdots \varphi_k c$ since $gc=0$. 

Since the support of $c$ is a Delone set of finite local complexity, the support of $\varphi_2 \cdots \varphi_k c$ must be a subset of a Delone set of finite local complexity.
Thus, by Lemma \ref{line polynomial annihilator implies periodicity} and Remark \ref{remark: subset FLC line polynomial annihilator}, the configuration $\varphi_2 \cdots \varphi_k c$  is periodic in direction $\bw_1$.
  
  The argument can be repeated for each $i \in \{ 2,\ldots, k\}$ and hence we conclude that there exist $r_1,\ldots,r_k \in \R$ such that the $\R^d$-polynomial
  $$
  (X^{r_1 \vec{w}_1}-1) \cdots (X^{r_k \vec{w}_k}-1)
  $$
  annihilates $c$.
\end{proof}

\begin{remark}
  More generally, we have again that the above theorem holds for any integral $\R^d$-configuration whose support is a subset of a Delone set of finite local complexity.
\end{remark}

\subsection{Periodic decomposition theorem for $\R^d$-configurations}
\label{sec:Decomposition theorem for generalized configurations with a special annihilator}

We now prove a periodic decomposition theorem for $\mathbb{R}^d$-configurations, extending the corresponding result from the $\mathbb{Z}^d$-setting.
The proof of this result is a direct generalization of the proof of Theorem \ref{thm: regular decomposition theorem} using Theorem \ref{main result 1}.

We begin with two lemmas as in the proof of Theorem \ref{thm: regular decomposition theorem} in \cite{fullproofs}.

\begin{lemma} \label{genlemma1}
    Let $\bv_1\in\R^d$ and $\bv_2\in\R^d$ be linearly independent vectors over $\Q$ such that the difference $\R^d$-polynomial $X^{\bv_2}-1$ annihilates a function $c' \in \Z^{\R^d}$. 
    There exists a function $c \in \Z^{\R^d}$ such that $(X^{\bv_1}-1)c=c'$ and $(X^{\bv_2}-1)c = 0$.
\end{lemma}

\begin{proof}
    The space $\R^d$ is partitioned to cosets modulo 
    $\langle \bv_1, \bv_2\rangle$.
    Let us fix an arbitrary $\vec{z} \in \R^d$ and consider the coset 
    $$
    \vec{z} + \langle \bv_1, \bv_2 \rangle
    = \{ \vec{z} + a \bv_1 + b \bv_2 \mid a,b \in \Z \}.
    $$
    
    The equation $(X^{\bv_1}-1)c=c'$ is satisfied in the coset 
    $\vec{z} + \langle \bv_1, \bv_2 \rangle$
    if and only if
    \begin{equation*}\label{eq:recurrence1}
    c_{\bz+(a-1)\bv_1 + b\bv_2} -c_{\bz+a\bv_1+b\bv_2} = c'_{a\bv_1+b\bv_2}
    \end{equation*}
    holds for every $a,b \in \Z$. 
    Set $c_{\vec{z} + b \vec{v}_2} = 0$. 
    Then the above equation defines the rest for all $a,b \in \Z$. 
    Since $\vec{z}$ was arbitrary, we can do this in every coset and obtain $c$ such that $(X^{\bv_1}-1)c=c'$ in every coset and hence $(X^{\bv_1}-1)c=c'$ everywhere.
    
    Finally, let us show that $(X^{\bv_2}-1) c = 0$. It suffices to show that $(X^{\bv_2}-1)c = 0$ in an arbitrary coset $\vec{z} + \langle \bv_1, \bv_2 \rangle$. This follows from the computation
    \begin{equation}\label{eq:recurrence2}
    (X^{\bv_1}-1)(X^{\bv_2}-1)c=(X^{\bv_2}-1)(X^{\bv_1}-1)c=(X^{\bv_2}-1)c'=0.
    \end{equation}
    Indeed, above we defined $c_{\vec{z} + b \bv_2} = 0$ for every $b \in \Z$.
    This implies that also $$[(X^{\bv_2}-1)c]_{\vec{z} + b \bv_2} = 0$$ for every $b \in \Z$. 
    From Equation \eqref{eq:recurrence2} it follows that $[(X^{\bv_2}-1)c]_{\vec{z} + a \vec{v}_1 + b \vec{v}_2} = 0$ for every $a,b \in \Z$. 
\end{proof}


\begin{lemma} \label{genlemma2}
    Let $c$ be an integral $\R^d$-configuration and let 
    $\bv_1,\ldots,\bv_m \in \R^d$ be pairwise linearly independent vectors over $\Q$.
    If the product $(X^{\bv_1}-1) \cdots (X^{\bv_m}-1)$ of difference $\R^d$-polynomials annihilates $c$, 
    then there exist functions $c_1, \ldots, c_m \in \Z^{\R^d}$ such that $(X^{\bv_i}-1) c_i = 0$ for each $i$ and 
    $$
    c = c_1 + \ldots + c_m.
    $$
\end{lemma}

\begin{proof}
    The proof is by induction on $m$.
    For $m=1$ the claim is clear.
    Assume then that $m \geq 2$ and that the claim holds for $m-1$.
    Since $(X^{\bv_m}-1) c$ is annihilated by $(X^{\bv_1}-1)\cdots (X^{\bv_{m-1}}-1)$, by the induction hypothesis there exist functions $c'_1, \ldots , c'_{m-1} \in \C^{\R^d}$ such that 
    $$
    (X^{\bv_m}-1) c = 
    c'_1 + \ldots + c'_{m-1}
    $$
    and each $c'_i$ is annihilated by $X^{\bv_i}-1$.
    By Lemma \ref{genlemma1} for each $i \in \{1,\ldots,m-1 \}$ there exists a function $c_i$ such that $(X^{\bv_m}-1) c_i = c'_i$ and $(X^{\bv_i}-1)c_i=0$.
    Set $c_m = c - c_1 - \ldots - c_{m-1}$.
    Then clearly $c = c_1+\ldots+c_m$ and moreover
    \begin{align*}
    (X^{\bv_m}-1) c_m &= (X^{\bv_m}-1) (c -c_1 - \ldots -c_{m-1}) \\
    & = (X^{\bv_m}-1) c - (X^{\bv_m}-1)c_1 - \ldots - (X^{\bv_m}-1)c_{m-1} \\
    & = c'_1 + \ldots + c'_{m-1} - c'_1 - \ldots - c'_{m-1} \\ &= 0.
    \end{align*}
    Hence, $c=c_1+\ldots+c_m$ and each $c_i$ is annihilated by $(X^{\bv_i}-1)$.
    The claim follows.
\end{proof}


We are now ready to state and prove the main decomposition result.

\begin{theorem}[Periodic decomposition theorem for $\R^d$-configurations] 
\label{thm:generalized decomposition theorem}
  Let $c$ be an integral $\R^d$-configuration with a non-trivial integral annihilator. Then there exist periodic functions $c_1, \ldots , c_m \in \Z^{\R^d}$ such that
  $$
  c = c_1 + \ldots + c_m.
  $$
\end{theorem}

\begin{proof}
By Theorem \ref{main result 1} the $\R^d$-configuration $c$ is annihilated by
$$
(X^{\vec{v}_1}-1)\cdots(X^{\vec{v}_m}-1)
$$
for some pairwise linearly independent $\bv_1,\ldots,\bv_m$ over $\Q$.
By Lemma \ref{genlemma2} there exist functions $c_1, \ldots , c_m \in \Z^{\R^d}$ such that their sum is $c$ and each $c_i$ is annihilated by $X^{\bv_i}-1$ and hence $\vec{v}_i$-periodic.
\end{proof}

\section{Meyer sets with slow patch complexity growth} \label{sec: meyer sets}

In this section, we study conditions under which annihilators arise for Delone sets. In particular, we focus on Meyer sets and show that sufficiently slow growth of the patch complexity forces the existence of non-trivial annihilators.
More precisely, we show that if a Meyer set $S$ satisfies $\liminf_{T \to \infty} \frac{N_S(T)}{T^d} = 0$, then it has low complexity with respect to some $\R^d$-shape.
Consequently, it has a non-trivial annihilator as we have seen in Section \ref{sec: delone configurations with annihilators}.

\subsection{From slow patch complexity growth to non-trivial annihilators}

We now develop a technical framework that allows us to relate patch complexity to pattern complexity.

Consider the set $U = (S+S-S) \cap B_R$ for a Meyer set $S$ with covering radius $R$.
The following lemma shows that this set is finite.


\begin{lemma}\label{lemma:the set U  is finite for Meyer}
  Let $S \subseteq \R^d$ be a Meyer set. 
  The set $S+S-S$ is locally finite.
\end{lemma}

\begin{proof}
  Without loss of generality we may assume that $S$ is translated such that $\vec{0} \in S$.
  Since $S$ is a Meyer set, there exists a finite set $F$ such that $S-S \subseteq S+F$.
  Using this twice and the fact that $\vec{0} \in S$ we get

  \begin{align*}
  S+S-S &\subseteq S+S-S-S = S-S - (S-S) \\
  &\subseteq S + F - (S + F) = S-S + F - F \\
  &\subseteq S + F + F - F.
  \end{align*}

  \noindent
  Lemma \ref{closed under addition with a finite set} says that the sum of a Delone set of finite local complexity and a non-empty finite set is also a Delone set of finite local complexity.
  In particular, $S + F + F - F$ is uniformly discrete since $F+F-F$ is a non-empty finite set.
  This implies that $S+S-S$ is also uniformly discrete as a subset of a uniformly discrete set.
  The claim follows since uniformly discrete sets are locally finite.
\end{proof}

Define for every $T \geq 0$ a set
$H_T = U \cup (S \cap B_T)$.
This is a finite set since $U$ is a finite set as seen above and since $S$ is a Delone set (and hence $S \cap B_T$ is a finite set).
We start by the observation that any non-empty $H_T$-pattern of $S$ contains an element of $S$ already in its $U$-part.

\begin{lemma} \label{lemma:Meyer H_T contains a point in its U-part}
  Let $S \subseteq \R^d$ be a Meyer set with covering radius $R$, and let $U = (S+S-S) \cap B_R$.
  Define for every $T \geq 0$ a set
  $H_T = U \cup (S \cap B_T)$.
  If $S \cap (H_T + \vec{t}) \neq \emptyset$, then also $S \cap (U + \vec{t}) \neq \emptyset$.
\end{lemma}

\begin{proof}
  Assume that $\vec{u} \in S \cap (H_T + \vec{t})$.
  If $\vec{u} \in U + \vec{t}$, then the claim is valid and we are done.
  So, assume that $\vec{u} \in (H_T \setminus U) + \vec{t}$.
  Since $H_T \setminus U \subseteq S$, we have $\vec{u} - \vec{t} \in S$.
  By the definition of $R$ we have $S \cap B_R(\vec{t}) \neq \emptyset$.
  So, let $\vec{w} \in S \cap B_R(\vec{t})$.
  We claim that $\vec{w} \in U + \vec{t}$, {\it i.e.}, $\vec{w} - \vec{t} \in U$.
  This follows from the observation that
  $$
  \vec{w}-\vec{t} = \vec{w} - \vec{u} + \vec{u} -\vec{t} = \vec{w} + (\vec{u} - \vec{t}) - \vec{u} \in S+S-S
  $$
  and from the fact that $\vec{w} - \vec{t} \in B_R$.
\end{proof}

Next, we obtain an upper bound for the pattern complexity $P_S(H_T)$ in terms of the patch-counting function.

\begin{lemma} \label{lemma:Meyer bound for the complexity of H_T}
  Let $S \subseteq \R^d$ be a Meyer set with covering radius $R$, and let $U = (S+S-S) \cap B_R$.
  Define for every $T \geq 0$ a set
  $H_T = U \cup (S \cap B_T)$.
  Then
  $$
  P_S(H_T) \leq 1 + |U| N_S(T + R).
  $$
\end{lemma}

\begin{proof}
  Let $P$ be a non-empty $H_T$-pattern of $S$, that is, $P = S \cap (H_T + \bt) \neq \emptyset$ for some $\bt \in \R^d$.
  Then by Lemma \ref{lemma:Meyer H_T contains a point in its U-part} it contains a point already in its $U$-part, that is, $\bu + \bt \in P$ for some $\bu \in U$.
  Clearly, $P$ is contained in the $(T+R)$-patch of $S$ centered at $\bu + \bt$.
    So, the $(T+R)$-patch centered at $\bu+\bt$ together with $\bu$ uniquely determines the $H_T$-pattern centered at $\bt$.
  Hence, the number of non-empty $H_T$-patterns up to translation is at most $|U| N_S(T + R)$.
  Thus,
  $$
  P_S(H_T) \leq 1 + |U| N_S(T + R)
  $$
  since the empty set is always in $\Patt{H_T}{S}$.
\end{proof}

Let us now analyze the size of $H_T$. 
It contains the set $S \cap B_T$.
Since $R$ is the covering radius of $S$, any ball of radius $R$ must contain at least one point of $S$.
Hence, the number of points of $S \cap B_T$ is at least the number of separate balls of radius $R$ contained in $B_T$.
We give a rough lower bound for $|S \cap B_T|$ which is also a lower bound for $|H_T|$.
Consider the hypercube
$$
C = \left \{ (u_1,\ldots,u_d) \mid |u_i| \leq \frac{T}{\sqrt{d}} \right \}
$$
contained in $B_T$ with sides of length $\frac{2T}{\sqrt{d}} $.
It is contained in $B_T$.
Write
$$
\frac{2T}{\sqrt{d}}  = k \cdot 2R + q
$$
where $k$ is a non-negative integer --- the integer part of $\frac{T}{\sqrt{d}R}$ --- and $q < 2R$ is a non-negative real number.
Clearly, $C$ contains at least $k^d$ separate hypercubes with sides of length $2R$ and hence at least $k^d$ separate balls of radius $R$.
Since $R$ is the covering radius of $S$, each ball of radius $R$ contains at least one point of $S$.
Thus,
$$
|H_T| \geq |S \cap B_T| \geq |S \cap C| \geq k^d = \left ( \frac{T}{\sqrt{d}R} - \frac{q}{2R}  \right )^d
$$
and hence $\liminf _{T \to \infty} \frac{|H_T|}{T^d} > 0$.


Now, we are ready show that sufficiently slow patch complexity growth implies low $\R^d$-pattern complexity.

\begin{theorem} \label{thm: low complexity Meyer}
  Let $S \subseteq \R^d$ be a Meyer set with covering radius $R$, and let $U = (S+S-S) \cap B_R$.
  Define for every $T \geq 0$ a set
  $H_T = U \cup (S \cap B_T)$.
  If
  $$
  \liminf_{T \rightarrow \infty} \frac{N_S(T)}{T^d}=0,
  $$
  then $S$ has low complexity with respect to $H_T$ for some $T$.
\end{theorem}

\begin{proof}
  By Lemma \ref{lemma:Meyer bound for the complexity of H_T} we have 
  $$
  P_S(H_T) \leq 1 + |U| N_S(T + R).
  $$
  Thus, if $\liminf_{T \to \infty} \frac{N_S(T)}{T^d} = 0$, then also 
  $\liminf_{T \to \infty} \frac{P_S(H_T)}{T^d} = 0$.
  As seen above, we have
  $\liminf _{T \to \infty} \frac{|H_T|}{T^d} > 0$.
  Thus, for some large $T$ we have
  $$
  P_S(H_T) \leq |H_T|.
  $$
\end{proof}

Together with the algebraic framework developed earlier, 
the above theorem yields the following important consequence.

\begin{corollary}\label{corollary: low complexity meyer annihilator}
  Let $S \subseteq \R^d$ be a Meyer set.
  If 
  $$
  \liminf_{T \to \infty} \frac{N_S(T)}{T^d} = 0,
  $$ 
  then $S$ has a non-trivial annihilator.
  In particular, $S$ has an annihilator of the form
  $$
  (X^{\bv_1}-1) \cdots (X^{\bv_m}-1).
  $$
\end{corollary}

\begin{proof}
  By Theorem \ref{thm: low complexity Meyer} the Meyer set $S$ has low complexity with respect to some $\R^d$-shape.
  Consequently, by Lemma \ref{lemma: non-trivial annihilator} it has a non-trivial annihilator and hence by Theorem \ref{main result 1} it has an annihilator of the desired form.
\end{proof}

\subsection{A related conjecture}

Finally, let us mention a related conjecture which is already known to be false.

\begin{conjecture}[Lagarias and Pleasants 2003 \cite{lagarias_pleasants_ETDS}]
Any non-periodic Delone set $S \subseteq \R^d$ satisfies
$$
\limsup_{T \to \infty} \frac{N_S(T)}{T^d} > 0.
$$
\end{conjecture}

\noindent
As mentioned, the conjecture has already proven to be false in \cite{cassaigne-example}
where it is given
for any $d \geq 3$ an example of
a non-periodic binary configuration $c \in \{0,1\}^{\Z^d}$ such that
$$
\lim_{n \to \infty} \frac{P_c(n)}{n^d} = 0
$$
where $P_c(n)=P_c(\{0,\ldots,n-1\}^d)$.
One can then define a $d$-dimensional Delone set $S \subseteq \Z^d$ such that $\bu \in S$ if and only if $c_{\bu}=1$.
It satisfies
$$\limsup_{T \to \infty} \frac{N_S(T)}{T^d} = 0.$$
So, this construction shows that the above conjecture does not hold.

However, if in the conjecture we strengthen the assumption that the Delone set $S$ is non-periodic and instead assume that it is a Meyer set and has no non-trivial annihilators, then the statement of the conjecture holds by Corollary \ref{corollary: low complexity meyer annihilator}.
Indeed, we have the following theorem.

\begin{theorem} \label{thm: meyer growth bound}
Any Meyer set $S \subseteq \R^d$ with no non-trivial annihilators satisfies
$$
\liminf_{T \to \infty} \frac{N_S(T)}{T^d} > 0.
$$
\end{theorem}

\begin{proof}
Let $S$ be a Meyer set with no non-trivial annihilators.
Then by Corollary \ref{corollary: low complexity meyer annihilator} we have
$$\liminf_{T \to \infty} \frac{N_S(T)}{T^d} > 0.$$
\end{proof}

\noindent
Thus, although slow patch complexity growth does not in general imply periodicity, our results show that it still enforces strong algebraic constraints in the form of annihilators.

\section{Forced periodicity of Delone configurations of finite local complexity} \label{sec: forced periodicity FLC}

In this section, we establish a statement of forced periodicity for Delone configurations of finite local complexity. 
We show that sufficiently strong algebraic constraints, in the form of annihilators with suitable geometric properties, imply global periodicity.
To formulate and prove our result, we first introduce some geometric notions.

For a fixed dimension $d$, we denote by $\G_k$ the set of all $k$-dimensional linear subspaces of $\R^d$.
Clearly, if $k > d$, then $\G_k = \emptyset$.


\begin{definition}
The \emph{open half space} in direction $\bv$ is the set
$$
H_{\bv} = \{ \bu \in \R^d \mid \bu \cdot \bv  < 0 \}.
$$
The \emph{closed half space} in direction $\bv$ is the set $\overline{H}_{\bv} = \{ \bu \in \R^d \mid \bu \cdot \bv \leq 0 \}$.
\end{definition}
\noindent
Note that
$$
\overline{H}_{\bv} \setminus H_{\bv} = \{ \bu \in \R^d \mid \bu \cdot \bv = 0 \} \in \G_{d-1}
$$
is the orthogonal space of $\bv$, and we may denote it by $\langle \bv \rangle ^{\perp}$.
See Figure \ref{fig:open half space} for a geometric illustration of these sets.

\begin{definition}
A non-empty finite set $F \Subset \R^d$ \emph{has a vertex} in direction $\bv$ if there exists a vector $\bt \in \R^d$ such that $F+ \bt \subseteq \overline{H}_{\bv}$ and $(F+\bt) \cap \langle \bv \rangle^{\perp} = \{\vec{0}\}$.
An $\R^d$-polynomial has a vertex in direction $\bv$ if its support has a vertex in direction $\bv$.
\end{definition}

\noindent
Intuitively, this means that after a suitable translation, the support of the polynomial lies entirely in a half-space, touching the boundary only at a single point.


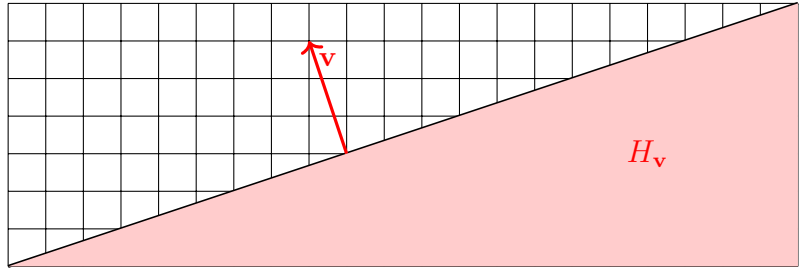
\begin{figure}[ht]
  \centering
   \begin{tikzpicture}[scale=0.5]
     \draw (0,0) grid (21,7);
     \draw[very thick] (0,0) -- (7*3,7);
     \draw[red,very thick,->] (9,3) -- (8,6);
     \node at (8.5,5.5) {\textcolor{red}{\small $\bv$}};
     \fill[red!20,opacity=0.5] (0,0) -- (21,7) -- (21,0) -- cycle;
     \node at (17,3) {\textcolor{red}{$H_{\bv}$}};
   \end{tikzpicture}
  \caption{The open half space $H_{\bv}$ in direction $\bv = (-1,3)$. The black line corresponds to the set $\overline{H}_{\bv} \setminus H_{\bv} = \langle \bv \rangle ^{\perp}$.}
  \label{fig:open half space}
\end{figure}

%

For a Delone configuration $c \in \A^{\R^d}$, let us denote
by
$$
\X_{c} = \{ \tau^{\bt}(c) \mid \bt \in \R^d, \tau^{\bt}(c)(\vec{0}) \neq 0 \}
$$
the set of translations of $c$ with non-zero values at the origin.

\begin{lemma}\label{lemma:hull}
Let $c$ be a Delone configuration.
If $c$ is not strongly periodic, then $\X_c$ is an infinite set.
\end{lemma}

\begin{proof}
Assume that $c$ is not strongly periodic.
Let $S$ be the support of $c$.

Assume first that $S$ is strongly periodic.
Thus, $S$ has $d$ linearly independent period vectors $\bt_1,\ldots,\bt_d$.
Since $c$ is not strongly periodic, one of these vectors, say $\bt_1$, is not a period of $c$.
Moreover, no multiple of $\bt_1$ is a period of $c$.
Let $e$ be a translation of $c$ such that $e \in \X_c$, that is, $e_{\vec{0}} \neq 0$.
Clearly, any period of $c$ is also a period of $e$ and vice versa.
Thus, for all $k \in \Z$ the Delone configurations $\tau^{k \bt_1}(e)$ are distinct.
Since $\bt_1$ is a period of the support of $e$ and $e_{\vec{0}} \neq 0$,  also $\tau^{k \bt_1}(e)(\vec{0}) \neq 0$ for all $k \in \Z$.
Thus, $\tau^{k \bt_1}(e) \in \X_c$ for all $k \in \Z$ and hence $X_c$ is an infinite set.

Assume then that $S$ is not strongly periodic.
It follows that the set
$\Lambda_S = \{\bt \in \R^d \mid S+\bt=S\}$ of periods of $S$ spans a vector space of dimension less than $d$.
By the relative denseness of $S$ there are infinitely many vectors $\bt_1,\bt_2,\ldots \in S$ such that the cosets $ \Lambda_S + \bt_i$ are distinct.
Let us define for every $i \in \Z_+$, a configuration $c_i =  \tau^{-\bt_i}(c)$.
We have $c_i(\vec{0}) = c(\bt_i) \neq 0$ since $\bt_i \in \supp(c)$.
So, we have $c_i \in \X_c$.
Let us show that for each $i\neq j$, also $c_i\neq c_j$ and hence $\X_c$ is an infinite set.
Assume on the contrary that $c_i=c_j$ for some $i \neq j$.
Then $\tau^{-\bt_i}(c) = \tau^{-\bt_j}(c)$ and hence $\tau^{\bt_j-\bt_i}(c) = c$.
So, $\bt_j-\bt_i$ is a period of $c$ and hence it is also a period of $S$.
Thus, $\bt_j-\bt_i \in \Lambda_S$ and hence $\bt_j \in \Lambda_S + \bt_i$.
Similarly, we conclude that $\bt_i \in \Lambda_S + \bt_j$ and hence $\Lambda_S + \bt_i = \Lambda_S + \bt_j$.
This is a contradiction.
\end{proof}

In the proof of our theorem of forced periodicity we need the following technical lemma.
We denote by $\mathbf{S}^1 = \{ \vec{v} \in \R^d \mid ||\vec{v}|| = 1 \}$ the $d$-dimensional unit sphere and use the fact that it is a compact subset of $\R^d$ under the usual Euclidean metric which implies that every sequence of elements of $\mathbf{S}^1$ has a converging subsequence.


\begin{lemma}\label{lemma:forced2}
Let $(\bv_n')_{n\geq 1}$ be a converging sequence of vectors $\bv_1',\bv_2',\ldots \in \mathbf{S}^1$, and let $\bv$ be the limit of this sequence.
Let $T_0 \in \R_+$ and $\bu\in \R^d$ be such that
$$
\bu \in B_{T_0}^{\circ}(-T_0 \bv).
$$
There exists $N_0 \in \Z_+$ such that
$$
\bu \in B_{T}^{\circ}(-T \bv'_{n})
$$
for all $T \geq T_0$ whenever $n \geq N_0$.
\end{lemma}

\begin{proof}
Let $\varepsilon= 1- \frac{d(-T_0 \bv, \bu)}{T_0}$.
This is a positive real number by the assumption $\bu \in B_{T_0}^{\circ}(-T_0 \bv)$.
Since $\bv = \lim_{n \to \infty} \bv_n'$, there exists $N_0$ such that
  $d(\bv_n',\bv) < \varepsilon$ for all $n \geq N_0$.  
By the triangle inequality we have
$$
  d(-T_0 \bv_{n}', \bu) \leq
  d(-T_0 \bv_{n}', -T_0 \bv) + d(-T_0 \bv, \bu)
  <
  T_0 \cdot \varepsilon + d(-T_0 \bv, \bu)
  = T_0
$$  
for all $n \geq N_0$.
Thus, $\bu \in B_{T_0}^{\circ}(-T_0 \bv_{n}')$ and hence
$\bu \in B_T^{\circ}(-T\bv_n')$ for all $T \geq T_0$ and $n \geq N_0$ since $B_{T_0}^{\circ}(-T_0 \bv_{n}') \subseteq B_{T}^{\circ}(-T \bv_{n}')$ for all $ T \geq T_0$.
\end{proof}

\noindent

The following technical lemma provides a way to detect differences between configurations at arbitrarily large scales.

\begin{lemma}\label{lemma: FLC arbitrarily large regions}
Let $c \in \A^{\R^d}$ be a Delone configuration of finite local complexity and assume that $\X_c$ is an infinite set.
Then there exists an unbounded increasing sequence $R_1,R_2,R_3,\ldots$ of positive real numbers
such that for each $n \in \Z_+$ there exist $e,e' \in \X_c$ such that
$$
e \restriction _{B_{R_n}^{\circ}} = e' \restriction _{ B_{R_n}^{\circ}}
$$
but
$$
e \restriction _{B_{R_n}} \neq e' \restriction _{ B_{R_n}}.
$$
\end{lemma}

\begin{proof}
Since $S = \supp(c)$ is a Delone set of finite local complexity, $N_S(T)$ is finite for all $T$.
For any $e \in \X_c$, the set $\supp(e) \cap B_T$ is a $T$-patch of $S=\supp(c)$.
Hence, there exist only finitely many different functions $e \restriction_{ B_T}$ for any $T$ because $\A$ is a finite set.

Assume on the contrary that the claim does not hold.
Then there exists $R_0$ such that  if
$e,e' \in \X_c$ with $e \restriction _{B_{R_0}} = e' \restriction _{ B_{R_0}}$, then $e=e'$.
It follows that $\X_c$ is finite 
since there are only finitely many functions $e \restriction_{ B_{R_0}}$.
\end{proof}


Now, we are ready to state and prove the main result of the section.

\begin{theorem} \label{thm: forced periodicity Delone FLC}
Let $c \in \A^{\R^d}$ be a Delone configuration of finite local complexity
and assume that for all non-zero $\bw$ it has an annihilator that has a vertex in direction $\bw$.
Then $c$ is strongly periodic.
\end{theorem}

\begin{proof}
The argument proceeds by contradiction. Assuming that the configuration is not strongly periodic, we construct configurations that disagree on larger and larger regions, and use the structure of annihilators to obtain a contradiction.
So, assume on the contrary that $c$ is not strongly periodic.
Thus, by Lemma \ref{lemma:hull} the set $\X_c$ is infinite.

Consequently, by Lemma \ref{lemma: FLC arbitrarily large regions} there exists
an unbounded sequence $0<R_1<R_2<R_3< \ldots$
such that for all $n\in\Z_+$
there exist a vector $\bv_{n} \in \R^d$ with $||\bv_n|| = R_n$ and 
$e,e' \in \X_c$ such that
$$
e \restriction _{ B_{R_n}^{\circ}} = e' \restriction_{ B_{R_n}^{\circ}}
$$
but
$$ 
e(\bv_{n}) \neq e'(\bv_{n}).
$$
  
  Consider the sequence $(\bv'_n)_{n \in \Z_+}$ where $\bv_n = R_n \bv'_n$ and $\bv'_{n} \in \mathbf{S}^1$.
  By compactness of
  $\mathbf{S}^1$
  we may assume that this sequence converges since we can replace it by a converging subsequence if necessary. 
  Let $\bv \in \mathbf{S}^1$ be the limit of this sequence. 
  
By the assumption, $c$ has an annihilator $f$ such that it has a vertex in direction $-\bv$. 
Without loss of generality we may assume that $f$ is multiplied by a suitable monomial $X^{\bu}$ such that $\vec{0} \in \supp(f)$ and $\supp(f) \setminus \{\vec{0}\} \subseteq H_{-\bv}$, {\it i.e.}, $-\supp(f) \setminus \{\vec{0}\} \subseteq H_{\bv}$.
Moreover, we may assume that
$f_{\vec{0}}=1$.
  Let $T_0$ be such that $-\supp(f) \subseteq B_{T_0}(-T_0 \bv)$ and
  $-\supp(f) \setminus \{\vec{0}\} \subseteq B_{T_0}^{\circ}(-T_0 \bv)$.

  Let then $N_0 \in \Z_+$ be such that 
  $$
  -\supp(f) \setminus \{\vec{0}\} \subseteq 
  B_{T}^{\circ}(-T \bv'_{n})
  $$
  for all $T \geq T_0$ and $n \geq N_0$.
  By Lemma \ref{lemma:forced2} such $N_0$ exists since $\supp(f) \setminus \{\vec{0}\}$ is a finite set.
  See Figure \ref{fig:proof} for an illustration.

  Let $n \geq N_0$ be such that $R_n \geq T_0$.
  So, we have
  $$
  -\supp(f) \setminus \{\vec{0}\} \subseteq 
  B_{R_n}^{\circ}(-R_n\bv_{n}') = B_{R_n}^{\circ}(-\bv_{n}).
  $$
  Now, consider two distinct $e,e' \in \X_c$ such that 
  $$
e \restriction _{ B_{R_n}^{\circ}} = e' \restriction_{ B_{R_n}^{\circ}}
$$
  but $e(\bv_{n}) \neq e'(\bv_{n})$.
  Since $-\supp(f) \setminus \{\vec{0}\} \subseteq B_{R_n}^{\circ}(- \bv_{n})$ and hence $\bv_{n} -\supp(f) \setminus \{\vec{0}\} \subseteq B_{R_n}^{\circ}$, it follows that 
$$
e \restriction _{\bv_{n} -\supp(f) \setminus \{\vec{0}\} } = e' \restriction _{\bv_{n} -\supp(f) \setminus \{\vec{0}\}}.
$$
We have $fe=0$ and $fe'=0$ and hence
  $$
  e(\bv_{n}) = -\sum_{\bu \in \supp(f) \setminus \{\vec{0}\}} f_{\bu} e(\bv_{n}-\bu)
  =-\sum_{\bu \in \supp(f) \setminus \{\vec{0}\}} f_{\bu} e'(\bv_{n}-\bu)
  = e'(\bv_{n}).
  $$
  This is a contradiction.
\end{proof}

\noindent
This result shows that sufficiently strong algebraic constraints on a Delone configuration force global periodic behavior.

\begin{figure}[ht]
  \centering
   \begin{tikzpicture}[scale=0.4]
     \draw (-0.5,-17.5) grid (24.5,7.5);
     \draw[very thick] (-0.5,-1/6) -- (7*3+3/2.5,7.5);
     \draw[red,thick,->] (9,3) -- (9-1/4,3+3/4);
     \node at (8.5,4) {\textcolor{red}{\small $\bv$}};
     \fill[red!20,opacity=0.5] (-0.5,-1/6) -- (22.5,7.5) -- (24.5,7.5) -- (24.5,-17.5) -- (-0.5,-17.5) -- (-0.5,-1/6);
     \node at (19,3) {\textcolor{red}{$H_{\bv}$}};
     \draw[very thick,green,fill=green!20, opacity=0.5](11,-3) circle (2*3.16227766);
     \draw[very thick,purple,fill=purple!20, opacity=0.5](9+1.5*2.828427,3-1.5*5.656854) circle (1.5*2*3.16227766);
     \draw[fill=green](11,-3) circle (2pt);
     \draw[fill=purple](9+1.5*2.828427,3-1.5*5.656854) circle (2pt);
     \draw[thick,blue] (9,3) -- (13,2) -- (12,-1) --(6.5,0) -- (9,3);
     \draw[green] (9,3) -- (11,-3);
     \draw[purple] (9,3) -- (9+1.5*2.828427,3-1.5*5.656854);
     \fill[blue!20,opacity=0.5] (9,3) -- (13,2) -- (12,-1) --(6.5,0) -- (9,3);
     \node at (10.1,1.3) {\textcolor{blue}{\tiny $-\supp(f)$}};
     \draw[fill=black](9,3) circle (3pt);
   \end{tikzpicture}
  \caption{Illustration of the proof of Theorem \ref{thm: forced periodicity Delone FLC}. The smaller ball is the ball $B_{T_0}(-T_0 \bv)$ and the bigger ball is the ball $B_{T}( -T \bv'_{n})$ for some $T \geq T_0$ and $n \geq N_0$.}
  \label{fig:proof}
\end{figure}
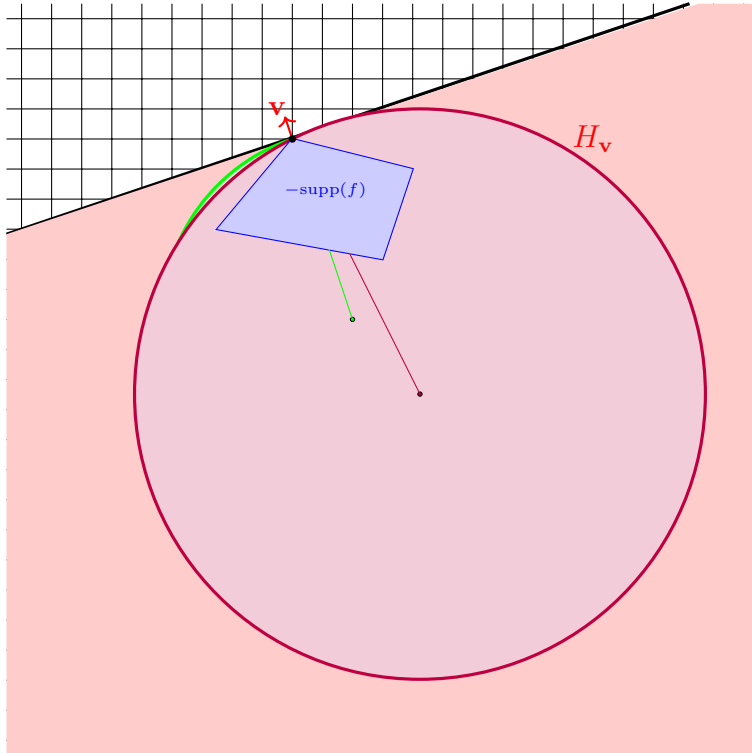

Theorem \ref{main result 1} together with the above theorem gives the following corollary.
The proof of the corollary resembles the proof of Lemma 3 in \cite{DLT_invited}.

\begin{corollary}\label{corollary: forced periodicity}
  Let $c$ be an integral Delone configuration of finite local complexity.
  Assume that for all $V \in \G_{d-1}$ it has a non-trivial annihilator $f$ such that $\supp(f) \cap V = \{\vec{0}\}$.
  Then $c$ is strongly periodic.
\end{corollary}

\begin{proof}
  Consider an arbitrary non-zero vector $\bv$ and let $V = \langle \bv \rangle^{\perp}$.
  By the assumption $c$ has an annihilator $f$ such that $\supp(f) \cap V = \{\vec{0}\}$.
  By Theorem \ref{main result 1} the $\R^d$-polynomial
  $$
  g = \prod_{\bu\in \supp(f) \setminus \{\vec{0}\}} (X^{k_{\bu}\bu} - 1)
  $$
  annihilates $c$ for some integers $k_{\bu}$.
  Since $(X^{k_{\bu}\bu} - 1) = -X^{k_{\bu}\bu}(X^{-k_{\bu}\bu} - 1)$, we may replace any $(X^{k_{\bu}\bu} - 1)$ by $(X^{-k_{\bu}\bu} - 1)$ if necessary and hence we may assume that
$k_{\bu} \bu \in H_{\bv}$ for each $\bu$.
Since the support of $g$ consists of $\vec{0}$ and sums of the vectors $  k_{\bu} \bu$ we have $\supp(g) \setminus \{\vec{0}\} \subseteq H_{\bv}$.
Thus, the annihilator $g$ of $c$ has a vertex in direction $\bv$ and hence the claim follows by Theorem \ref{thm: forced periodicity Delone FLC}.
\end{proof}

\noindent
In particular, we have the following known result as a corollary.
The original proof of the result uses Theorem \ref{special annihilator configurations} and the theory of expansive subspaces by Boyle and Lind \cite{boyle-lind}.

\begin{corollary}[\cite{DLT_invited}]
  Let $c$ be an integral configuration.
  Assume that for all $V \in \G_{d-1}$ it has a non-trivial annihilator $f$ such that $\supp(f) \cap V = \{\vec{0}\}$.
  Then $c$ is strongly periodic.
\end{corollary}

\begin{remark}
The converse direction of Corollary \ref{corollary: forced periodicity} holds for any $c \in \C^{\R^d}$.
In other words, a strongly periodic function $c \in \C^{\R^d}$ has for all $V \in \G_{d-1}$ a periodizer $f$ such that $\supp(f) \cap V = \{\vec{0}\}$.
Indeed, since $c$ has $d$ linearly independent period vectors $\bv_1,\ldots,\bv_d$, it follows that for any $V \in \G_{d-1}$ some $\bv_i$ is not in $V$ and hence $X^{\bv_i}-1$ is an annihilator (and hence a periodizer) of $c$ satisfying $\supp(X^{\bv_i}-1)\cap V = \{\vec{0}\}$.
\end{remark}


\section{Conclusion and further directions} \label{sec: conclusion}

In this work, we studied the relationship between local complexity and periodicity of Delone configurations using an algebraic framework based on annihilators. By extending methods from multidimensional symbolic dynamics to $\R^d$-configurations, we established a connection between geometric constraints on configurations and algebraic relations between their translations.

Our main results show that the existence of non-trivial annihilators imposes strong algebraic and geometric structure, including long-range dependancies. 
In particular, we proved that such annihilators can be chosen in a simple product form and for configurations of finite local complexity, with directions that are pairwise linearly independent over $\R$. As a consequence, configurations with annihilators admit decompositions into finitely many components that are periodic in pairwise rationally independent directions. This provides a direct link between algebraic relations among translations and geometric structure of Delone multisets.

We further demonstrated that sufficiently slow patch complexity growth naturally leads to the existence of annihilators in the setting of Meyer sets.
Moreover, we studied forced periodicity of Delone configurations of finite local complexity.
Altogether, our results show that local complexity gives rise to rigid algebraic constraints which in some settings enforce global periodicity.

There are several interesting directions for future work. One natural question is to develop more analogues of results from symbolic dynamics, such as Nivat-type theorems, in the setting of Delone sets, where patch complexity replaces rectangular pattern complexity (or more generally pattern complexity with respect to convex shapes). It would also be of interest to further explore the relationship between annihilators and geometric properties of point sets beyond finite local complexity.


\subsection*{Acknowledgements}

Both authors were supported by the Research Council of Finland grant 354965.


\section*{References}
\addcontentsline{toc}{section}{References}

\bibliographystyle{abbrv}
\bibliography{biblio}

\end{document}